\documentclass[3p,10pt,a4paper,twoside,fleqn,sort&compress]{amsart}
\usepackage{amssymb,amsmath,latexsym}
\newtheorem{theorem}{Theorem}[section]

\newtheorem{corollary}[theorem]{Corollary}

\newtheorem{definition}[theorem]{Definition}

\newtheorem{lemma}[theorem]{Lemma}

\newtheorem{proposition}[theorem]{Proposition}
\newtheorem{remark}[theorem]{Remark}

\begin{document}
	
	%
\title[On upper triangular operator matrices over $C^{*}$-algebras]{On upper triangular operator matrices over $C^{*}$-algebras}
	
\author[Stefan Ivkovi\'{c}]{Stefan Ivkovi\'{c}}
\address{Stefan Ivkovi\'{c} ,The Mathematical Institute of the Serbian Academy of Sciences and Arts, Kneza Mihaila 36 p.p. 367, 11000 Beograd, Serbia, Tel.: +381-69-774237 }
\email{stefan.iv10@outlook.com}
	\newcommand{\AuthorNames}{Stefan Ivkovi\'{c}}
	
	\newcommand{\FilMSC}{Primary xxxxx (mandatory); Secondary xxxxx, xxxxx (optionally)}
	\newcommand{\FilKeywords}{(keywords, mandatory)}
	\newcommand{\FilCommunicated}{(name of the Editor, mandatory)}
	\newcommand{\FilSupport}{Research supported by ... (optionally)}

	\begin{abstract}
			In this paper we study  the operator matrices 
		\begin{center}
			$\mathbf{M}_{\mathrm{C}}^{\mathcal{A}}=\left\lbrack
			\begin{array}{ll}
			F & C \\
			0 & D \\
			\end{array}
			\right \rbrack
			$
		\end{center}
		acting on $H_{\mathcal{A}} \oplus H_{\mathcal{A}}, $ where $F,D \in B^{a}(H_{\mathcal{A}}).$ We investigate the relationship  between the semi-Fredholm properties of $\mathbf{M}_{\mathrm{C}}^{\mathcal{A}} $ and of $F,D,$ when $F,D$ are fixed and $C$ varies over $B^{a}(H_{\mathcal{A}}),$ as an analogue of the results by Djordjevi\'{c}. 
	\end{abstract}
	\maketitle
	\makeatletter
	\renewcommand\@makefnmark%
	{\mbox{\textsuperscript{\normalfont\@thefnmark)}}}
	\makeatother
\vspace{10pt}
\vspace{10pt}
MSC (2010): Primary MSC 47A53 and Secondary MSC 46L08	\\	
	\section{Introduction }
	Perturbations of spectra of operator matrices were earlier studied in several papers such as \cite{DDj}.
	In \cite{DDj} Djordjevic lets $X$ and $Y$ be Banach spaces and the operator $\mathbf{M}_{\mathrm{C}}:X \oplus Y \longrightarrow X \oplus Y $ be given as $2\times2$ operator matrix
	$\left\lbrack
	\begin{array}{ll}
	\mathrm{A}& \mathrm{C} \\
	0 & \mathrm{B} \\
	\end{array}
	\right \rbrack
	$
	where $\mathrm{A}\in B(X),$ $\mathrm{B}\in B(Y)$ and $\mathrm{C} \in B(Y,X).$ Djordjevic investigates the relationship between certain semi-Fredholm properties of $\mathrm{A},\mathrm{B}$ and certain semi-Fredholm properties of $\mathbf{M}_{\mathrm{C}}$. Then he deduces as corollaries the description of the intersection of spectra of ${\mathbf{M}_{\mathrm{C}}}^{\prime}s,$ when $\mathrm{C}$ varies over all operators in $B(Y,X)$ and $\mathrm{A},\mathrm{B}$ are fixed, in terms of spectra of $\mathrm{A}$ and $\mathrm{B}.$ The spectra which he considers are not in general ordinary spectra, but rather different kind of Fredholm spectra such as essential spectra, left and right Fredholm spectra etc...
	
	Some of the main results in \cite{DDj} are Theorem 3.2, Theorem 4.4 and Theorem 4.6. In Theorem 3.2 Djordjevic gives necessary and sufficient conditions on operators A and B for the operator $M_{C}$ to be Fredholm. 
	
	Recall that two Banach spaces $U$ and $V$ are \textit{isomorpphic up to a finite dimensional subspace,} if one following statements hold: 
	\begin{description}
		\item[(a)]{there exists a bounded below operator $J_{1}: U \Rightarrow V$, such that dim $V/J_{1}(U) < \infty ,$ or };
		\item[(b)]{there exists a bounded below operator $J_{2}: V \Rightarrow U$, such that dim $V/J_{2}(V) < \infty .$}
	\end{description}
	Recall also that for a Banach space $X,$ the sets $\Phi(X),\Phi_{l}(X),\Phi_{r}(X) $ denote the sets of all Fredholm, left-Fredholm and right-Fredholm operators on $X,$ respectively. 
	\begin{theorem} \label{t3.2} \cite[Theorem 3.2]{DDj}
		Let $A \in \mathcal{L}(X)$ and $B \in \mathcal{L}(Y)$ be given and consider the statments:
		\begin{description}
			\item[(i)]{$M_{C} \in \Phi(X \oplus Y) $ for some $C \in \mathcal{L}(Y,X)$;}
			\item[(ii)]
			\begin{description}
				\item[(a)]{$A \in \Phi_{l}(X)$};
				\item[(b)]{$B \in \Phi_{r}(Y)$};
				\item[(c)]{$\mathcal{N}(B)$ and $X/\overline{\mathcal{R}(A)}$ are isomorphic up to a finite dimensional suspace.}
			\end{description}
		\end{description}
		Then (i) $\Leftrightarrow$ (ii).
	\end{theorem}	
	The implication (i) $\Rightarrow $ (ii) was proved in \cite{HLL}, whereas Djordjevic proves the implication (ii) $\Rightarrow $ (i).
	
	Similarly in Theorem 4.4 and Theorem 4.6 of \cite{DDj} Djordjevic investigates the case when $M_{C}$ is right and left semi-Fredholm operator, respectively. 
	Here we are going to recall these results as well, but first we repeat the following definition from \cite{DDj}:
	\begin{definition} \label{d380n}   \cite[Definition 4.2]{DDj}
		Let $X$ and $Y$ be Banach spaces. We say that $X$ can be embeded in $Y$ and write $X \preceq Y $ if and only if there exists a left invertible operator $J:X \rightarrow Y.$ We say that $X$ can essentially be embedded in $Y$ and write  $X \prec Y ,$ if and only if  $X \preceq Y $ and $Y/T(X)$ is an infinite dimensional linear space for all $T \in  \mathcal{L}(X,Y).$
	\end{definition}
	\begin{remark} \cite[Remark 4.3]{DDj}
		Obviously,  $X \preceq Y $ if and only if there exists a right invertible operator $J_{1}:Y \rightarrow X$.
	\end{remark}
	If $H$ and $K$ are Hibert spaces, then $H \preceq K$ if and only if dim $H \leq $ dim $K.$ Also $H \prec K$ if and only if dim $H <  $ dim $K$ and $K$ is ifinite dimensional. Here dim $H$ denotes the orthogonal dimension of $H.$
	\begin{theorem} \label{t280n}  
		\cite[Theorem 4.4]{DDj} Let $A \in \mathcal{L}(X)$ and $B \in \mathcal{L}(Y)$ be given operators and consider the following statements:
		\begin{description}
			\item[(i)]
			\begin{description}
				\item[(a)]{$B \in \Phi_{r}(Y)$};
				\item[(b)]{($A \in \Phi_{r}(X)),$ or ($\mathcal{R}(A)$ is closed and complemented in $X$ and\\
					$X/{\mathcal{R}(A)} \preceq \mathcal{N}(B)).$ }
			\end{description}	
			\item[(ii)] {$M_{C } \in \Phi_{r}(X \oplus Y ) $ for some $C \in \mathcal{L}(X,Y).$}
			\item[(iii)]
			\begin{description}
				\item[(a)]{$B \in \Phi_{r}(Y)$};
				\item[(b)]{$A \in \Phi_{r}(X)$ or ($\mathcal{R}(A)$ is not closed, or $\mathcal{N}(B) \prec X/ \overline{\mathcal{R}(A)}$ does not hold. }
			\end{description}
		\end{description}
	\end{theorem}
	Then (i) $\Leftrightarrow$ (ii)  $\Leftrightarrow$ (iii).
	\begin{theorem} \label{t380n}  \cite[Theorem 4.6]{DDj}
		Let $A \in \mathcal{L}(X)$ and $B \in \mathcal{L}(Y)$ be given operators and consider the following statements:
		\begin{description}
			\item[(i)]
			\begin{description}
				\item[(a)]{$A \in \Phi_{l}(Y)$};
				\item[(b)]{($B \in \Phi_{l}(X)),$ or ($\mathcal{R}(B)$ and $\mathcal{N}(B)$  are closed and complemented subspaces of $Y$ and $ \mathcal{N}(B)) \preceq  X/\overline{\mathcal{R}(A)}.$ }
			\end{description}	
			\item[(ii)] {$M_{C } \in \Phi_{l}(X \oplus Y ) $ for some $C \in \mathcal{L}(Y,X).$}
			\item[(iii)]
			\begin{description}
				\item[(a)]{$A \in \Phi_{l}(X)$};
				\item[(b)]{$B \in \Phi_{l}(X),$ or ($\mathcal{R}(B)$ is not closed, or $\mathcal{R}(A)^{\circ} \prec \mathcal{N}(B)$ does not hold. }
			\end{description}
		\end{description}
	\end{theorem}
	Then (i) $\Leftrightarrow$ (ii)  $\Leftrightarrow$ (iii).

	Now, Hilbert $C^{*}$-modules are natural generalization of Hilbert spaces when the field of scalars is replaced by a $C^{*}$-algebra. 
	
	Fredholm theory on Hilbert $C^*$-modules as a generalization of Fredholm theory on Hilbert spaces was started by Mishchenko and Fomenko in \cite{MF}. They have elaborated the notion of a Fredholm operator on the standard module $H_{\mathcal{A}} $ and proved the generalization of the Atkinson theorem. Their definition of $\mathcal{A}$-Fredholm operator on $H_{\mathcal{A}}$ is the following:
	
	\cite[Definition ]{MF} A (bounded $\mathcal{A}$ linear) operator $F: H_{\mathcal{A}} \rightarrow H_{\mathcal{A}}$ is called $\mathcal{A}$-Fredholm if\\
	1) it is adjointable;\\
	2) there exists a decomposition of the domain $H_{\mathcal{A}}= M_{1} \tilde{\oplus} N_{1} ,$ and the range, $H_{\mathcal{A}}= M_{2} \tilde{\oplus} N_{2}$, where $M_{1},M_{2},N_{1},N_{2}$ are closed $\mathcal{A}$-modules and $N_{1},N_{2}$ have a finite number of generators, such that $F$ has the matrix from 
	\begin{center}
		$\left\lbrack
		\begin{array}{ll}
		F_{1} & 0 \\
		0 & F_{4} \\
		\end{array}
		\right \rbrack
		$
	\end{center}
	with respect to these decompositions and $F_{1}:M_{1}\rightarrow M_{2}$ is an isomorphism.\\
	The notation $\tilde{ \oplus} $ denotes the direct sum of modules without orthogonality, as given in \cite{MT}.\\
	\\
	In \cite{I} we vent further in this direction and defined semi-$\mathcal{A}$-Fredholm operators on Hilbert $C^{*}$-modules. We investigated then and proved several properties of these generalized semi Fredholm operators on Hilbert $C^{*}$-modules as an analogue or generalization of the well-known properties of classical semi-Fredholm operators on Hilbert and Banach spaces.\\
	
	In particular we have shown that the class of upper semi-$\mathcal{A}$-Fredholm operators and lower semi-$\mathcal{A}$-Fredholm operators on $H_{\mathcal{A}},$ denoted by $\mathcal{M} \Phi_{+}(H_{\mathcal{A}})$ and\\
	$\mathcal{M} \Phi_{-}(H_{\mathcal{A}}),$ respecively, are exactly those that are one-sided invertible modulo compact operators on $H_{\mathcal{A}}$. Hence they are natural generalizations of the classical left and right semi-Fredholm operators on Hilbert spaces.
	
	The idea in this paper was to use these new classes semi-$\mathcal{A}$-Freholm of operators on $H_{\mathcal{A}}$ and prove that an analogue or a generalized version of \cite[Theorem 3.2]{DDj}, \cite[Theorem 4.4]{DDj}, \cite[Theorem 4.6]{DDj} hold when one considers these new classes of operators. We let $B^{a}(H_{\mathcal{A}})$ denote the set of all bounded, adjointable operators on $H_{\mathcal{A}}$ and we consider $M_{C}^{\mathcal{A}} : H_{\mathcal{A}} \oplus H_{\mathcal{A}} \rightarrow  H_{\mathcal{A}} \oplus H_{\mathcal{A}} $ given as $ 2\times 2 $ operator matrix 
	$M_{C}^{\mathcal{A}}= 
	\left\lbrack
	\begin{array}{ll}
	\mathrm{F}& \mathrm{C} \\
	0 & \mathrm{D} \\
	\end{array}
	\right \rbrack
	,$ where $F,C,D \in B^{a}(H_{\mathcal{A}}).$ Using this set up and these generalized classes of $\mathcal{A}$-Fredholm and semi-$\mathcal{A}$-Fredholm operators on $H_{\mathcal{A}}$ defined in \cite{I}, \cite{MF}, we obtain generalizations of Theorem 3.2, Theorem 4.4 and Theorem 4.6 in \cite{DDj}. Actually, our Theorem 3.2 is a generalization of a result in \cite{HLL}, as the implication in one way in Theorem 3.2 in \cite{DDj} was already proved in \cite{HLL}.
	
	In addition, we show that in the case when $X=Y=H$ where $H$ is a Hilbert space Theorem 4.4 and Theorem 4.6 in \cite{DDj} can be simplified. 
	\\
	
	Let us remind now the definition of the essential spectrum of bounded operators on Banach spaces. Namely, for a bounded operator T on a Banach space, the essential spectrum of T denoted $\sigma_{e}(T)$ is defined to be the set of all $\lambda \in \mathbb{C}$ for which $T-\lambda I$ is not Fredholm.
	\\
	In  \cite{DDj} Djordjevic considers the essential spectra of $A,B,M_{C}$ and he describes the situation when $\sigma_{e}(M_{C})=\sigma_{e}(A) \cup \sigma_{e}(B)$ in a chain of propositions. He shows first in Proposition 3.1 that $\sigma_{e}(M_{C}) \subset \sigma_{e}(A) \cup \sigma_{e}(B) $ in general and then, in Proposition 3.5 he gives sufficient conditions on A and B for the equality to hold. 
	
	Next, passing from Hilbert space to Hilbert $C^{*}$-modules we don't only replace the field of scalars by a $C^{*}$-algebra $\mathcal{A},$ but also work with $Z(\mathcal{A})$ valued spectrum instead of the standard one. Namely, given an $\mathcal{A}$-linear, bounded, adjointable operator $F$ on $H_{\mathcal{A}},$ we consider the operators of the form $F-\alpha 1$ as $\alpha$ varies over $Z(\mathcal{A})$ and this gives rise to a different kind of spectra of $F$ in $Z(\mathcal{A})$ as a generalization of ordinary spectra of $F$ in $\mathbb{C} .$  Using the generalized definitions of Fredholm and semi-Fredholm operators on $H_{\mathcal{A}}$ given in \cite{MF} and \cite{I} together with these new, generalized spectra in $Z(\mathcal{A})$. Finally we give a description of the intersection, when $C$ varies over $B^{a}(H_{\mathcal{A}}),$ of generalized essential spectra in $Z(\mathcal{A}),$ of the operator matrix $M_{C}^{\mathcal{A}}.$ We deduce this description as corollary from our generalizations of Theorem 3.2 in \cite{DDj}. Similar corollaries follows from our generalizations of Theorem 4.4 and Theorem 4.6 in  \cite{DDj}, however in these corollaries we consider generalized left and right Fredholm spectra of $M_{C}^{\mathcal{A}}$ instead of generalized essential spectrum of $M_{C}^{\mathcal{A}}.$
	
	\section{Preliminaries}
	In this section we are going to introduce the notation, the definitions in \cite{I} that are needed in this paper as well as some auxiliary results which are going to be used later in the proofs.	Throughout this paper we let $\mathcal{A}$ be a unital $\mathrm{C}^{*}$-algebra, $H_{\mathcal{A}}$ be the standard module over $\mathcal{A}$ and we let $B^{a}(H_{\mathcal{A}})$ denote the set of all bounded , adjointable operators on $H_{\mathcal{A}}.$ Next, for the $\mathrm{C}^{*}$-algebra $\mathcal{A},$ we let $Z(\mathcal{A})= \lbrace \alpha \in \mathcal{A} \mid \alpha \beta = \beta \alpha $ for all $\beta \in \mathcal{A} \rbrace $ and for $\alpha \in Z(\mathcal{A})$ we let $ \alpha \mathrm{I}$ denote the operator from $H_{\mathcal{A}}$  into $H_{\mathcal{A}}$ given by $\alpha \mathrm{I} (x) $ for all $x \in H_{\mathcal{A}}.$ The operator  $ \alpha \mathrm{I}$ is obviously $\mathcal{A}$-linear since $\alpha \in Z(\mathcal{A}) $ and it is adjointable with its adjoint $ \alpha^{*} \mathrm{I}.$ According to \cite[ Definition 1.4.1] {MT}, we say that a Hilbert $\mathrm{C}^*$-module $M$ over $\mathcal{A}$ is finitely generated if there exists a finite set $ \lbrace x_{i} \rbrace \subseteq M $  such that $M $ equals the linear span (over $\mathrm{C}$ and $\mathcal{A} $) of this set.\\
	\begin{definition}\label{d210n} \cite[Definition 2.1]{I} 
		Let $\mathrm{F} \in B^{a}(H_{\mathcal{A}}).$ 
		We say that $\mathrm{F} $ is an upper semi-{$\mathcal{A}$}-Fredholm operator if there exists a decomposition 
		$$H_{\mathcal{A}} = M_{1} \tilde \oplus N_{1} \stackrel{\mathrm{F}}{\longrightarrow}   M_{2} \tilde \oplus N_{2}= H_{\mathcal{A}}$$
		with respect to which $\mathrm{F}$ has the matrix\\
		\begin{center}
			$\left\lbrack
			\begin{array}{ll}
			\mathrm{F}_{1} & 0 \\
			0 & \mathrm{F}_{4} \\
			\end{array}
			\right \rbrack,
			$
		\end{center}
		where $\mathrm{F}_{1}$ is an isomorphism $M_{1},M_{2},N_{1},N_{2}$ are closed submodules of $H_{\mathcal{A}} $ and $N_{1}$ is finitely generated. Similarly, we say that $\mathrm{F}$ is a lower semi-{$\mathcal{A}$}-Fredholm operator if all the above conditions hold except that in this case we assume that $N_{2}$ ( and not $N_{1}$ ) is finitely generated.	
	\end{definition}
	Set
	\begin{center}
		$\mathcal{M}\Phi_{+}(H_{\mathcal{A}})=\lbrace \mathrm{F} \in B^{a}(H_{\mathcal{A}}) \mid \mathrm{F} $ is upper semi-{$\mathcal{A}$}-Fredholm $\rbrace ,$	
	\end{center}
	\begin{center}
		$\mathcal{M}\Phi_{-}(H_{\mathcal{A}})=\lbrace \mathrm{F} \in B^{a}(H_{\mathcal{A}}) \mid \mathrm{F} $ is lower semi-{$\mathcal{A}$}-Fredholm $\rbrace ,$	
	\end{center}
	$\mathcal{M}\Phi(H_{\mathcal{A}})=\lbrace \mathrm{F} \in B^{a}(H_{\mathcal{A}}) \mid \mathrm{F} $ is $\mathcal{A}$-Fredholm operator on $H_{\mathcal{A}}\rbrace .$
	
	\begin{lemma} \label{l210}   
		Let $M,N,W$ be Hilbert $\mathrm{C}^{*}$-modules over a unital $\mathrm{C}^{*}$-algebra $\mathcal{A}.$ If $\mathrm{F} \in B^{a}(M,N),\mathrm{D} \in B^{a}(N,W)$ and $\mathrm{D}\mathrm{F} \in \mathcal{M} \Phi(M,W)$, then there exists a chain of decompositions
		$$M=M_{2}^{\bot} \oplus M_{2}\stackrel{\mathrm{F}}{\longrightarrow} \mathrm{F}(M_{2}^{\bot}) \oplus R\stackrel{\mathrm{D}}{\longrightarrow}   W_{1} \tilde{\oplus} W_{2} =W$$ 
		w.r.t. which $\mathrm{F},\mathrm{D}$ have the matrices
		$\left\lbrack
		\begin{array}{ll}
		\mathrm{F}_{1} & 0 \\
		0 & \mathrm{F}_{4} \\
		\end{array}
		\right \rbrack
		,$ 
		$\left\lbrack
		\begin{array}{ll}
		\mathrm{D}_{1} & \mathrm{D}_{2} \\
		0 & \mathrm{D}_{4} \\
		\end{array}
		\right \rbrack,
		$
		respectively, where $\mathrm{F}_{1}, \mathrm{D}_{1} $ are isomorphisms, $M_{2}, W_{2}$ are finitely generated, $\mathrm{F}(M_{2}^{\bot}) \oplus R=N$ and in addition  $M=M_{2}^{\bot} \oplus M_{2}\stackrel{\mathrm{D}\mathrm{F}}{\longrightarrow} W_{1} \tilde{\oplus} W_{2} =W$  is an $ \mathcal{M} \Phi$-decomposition for $\mathrm{D}\mathrm{F}$.	
	\end{lemma}
	\begin{proof}
		By the proof of \cite[Theorem 2.7.6 ]{MT} applied to the operator $$\mathrm{D}\mathrm{F} \in  \mathcal{M} \Phi (M,W),$$ there exists an $\mathcal{M} \Phi$-decomposition $$M=M_{2}^{\bot} \oplus M_{2}\stackrel{\mathrm{D}\mathrm{F}}{\longrightarrow} W_{1} \tilde{\oplus} W_{2} =W$$ for $\mathrm{D}\mathrm{F}.$ This is because the proof of \cite[Theorem 2.7.6 ]{MT} also holds when we consider arbitrary Hilbert $\mathrm{C}^{*}$-modules $M$ and $W$ over unital $\mathrm{C}^{*}$-algebra $\mathcal{A}$ and not only the standard module $ H_{\mathcal{A}} $. Then we can proceed as in the proof of Theorem 2.2  \cite{I}, part $2) \Rightarrow 1).$	
	\end{proof}
	\begin{lemma} \label{l220}  
		If  $\mathrm{D} \in  \mathcal{M} \Phi_{-} (H_{\mathcal{A}})$, then there exists an $ \mathcal{M} \Phi_{-}$-decompo- sition $ H_{\mathcal{A}} = {N_{1}^{\prime}}^{\perp} \oplus {N_{1}^{\prime}}\stackrel{\mathrm{D}}{\longrightarrow} M_{2} \oplus {N_{2}^{\prime}} = H_{\mathcal{A}} $ for $\mathrm{D}.$ Similarly, if $\mathrm{F} \in  \mathcal{M} \Phi_{+} (H_{\mathcal{A}}),$ then there exists an $\mathcal{M} \Phi_{+}$-decomposition $H_{\mathcal{A}} = M_{2}^{\bot} \oplus N_{1}\stackrel{\mathrm{F}}{\longrightarrow} N_{2}^{\perp} \oplus N_{2}=H_{\mathcal{A}}$ for $\mathrm{F}.$
	\end{lemma}
	\begin{proof}
		Follows from the proofs of Theorem 2.2  \cite{I} and  Theorem 2.3 \cite{I}, part $1) \Rightarrow 2)  .$
	\end{proof}
	
	\begin{definition} \label{d280n}  \cite[Definition 5.1]{I}
		Let $\mathrm{F} \in \mathcal{M}\Phi (H_{\mathcal{A}})$. 
		We say that $\mathrm{F} \in \tilde {\mathcal{M}} \Phi_{+}^{-} (H_{\mathcal{A}})$ if there exists a decomposition 
		$$H_{\mathcal{A}} = M_{1} \tilde \oplus {N_{1}}\stackrel{\mathrm{F}}{\longrightarrow} M_{2} \tilde \oplus N_{2}= H_{\mathcal{A}} $$
		with respect to which $\mathrm{F}$ has the matrix
		\begin{center}
			$\left\lbrack
			\begin{array}{ll}
			\mathrm{F}_{1} & 0 \\
			0 & \mathrm{F}_{4} \\
			\end{array}
			\right \rbrack,
			$
		\end{center}
		where $\mathrm{F}_{1}$ is an isomorphism, $N_{1},N_{2}$ are closed, finitely generated and $N_{1} \preceq N_{2},$ that is $N_{1}$ is isomorphic to a closed submodule of $N_{2}$. We define similarly the class $\tilde {\mathcal{M}}\Phi_{-}^{+} (H_{\mathcal{A}})$, the only difference in this case is that $N_{2} \preceq N_{1}$. Then we set
		$$\mathcal{M}\Phi_{+}^{-} (H_{\mathcal{A}})= (\tilde {\mathcal{M}} \Phi_{+}^{-} (H_{\mathcal{A}})) \cup (\mathcal{M}\Phi_{+} (H_{\mathcal{A}}) \setminus \mathcal{M}\Phi (H_{\mathcal{A}}))$$ 
		and
		$$\mathcal{M}\Phi_{-}^{+} (H_{\mathcal{A}})= (\tilde {\mathcal{M}}\Phi_{-}^{+} (H_{\mathcal{A}})) \cup (\mathcal{M}\Phi_{-} (H_{\mathcal{A}}) \setminus \mathcal{M}\Phi (H_{\mathcal{A}}))$$
	\end{definition}
	
	\section{Perturbations of spectra in $Z(\mathcal{A})$ of operator matrices acting on $H_{\mathcal{A}} \oplus H_{\mathcal{A}} $}
	It this section we will consider the operator ${\mathbf{M}_{\mathrm{C}}^{\mathcal{A}}}(\mathrm{F},\mathrm{D}):H_{\mathcal{A}} \oplus H_{\mathcal{A}} \rightarrow H_{\mathcal{A}}  \oplus H_{\mathcal{A}}$ given as $2 \times 2$ operator matrix \begin{center}
		$\left\lbrack
		\begin{array}{ll}
		\mathrm{F} & \mathrm{C} \\
		0 & \mathrm{D} \\
		\end{array}
		\right \rbrack
		,$
	\end{center} where $\mathrm{C} \in B^{a}  (H_{\mathcal{A}})$.\\
	To simplify notation, throughout this paper, we will only write $ {\mathbf{M}_{\mathrm{C}}^{\mathcal{A}}}$ instead of ${\mathbf{M}_{\mathrm{C}}^{\mathcal{A}}} (\mathrm{F},\mathrm{D})$ when $\mathrm{F},\mathrm{D} \in B^{a}  (H_{\mathcal{A}})$ are given.\\
	Let $\sigma_{e}^{\mathcal{A}}({\mathbf{M}_{\mathrm{C}}^{\mathcal{A}}})=\lbrace \alpha \in Z(\mathcal{A}) | {\mathbf{M}_{\mathrm{C}}^{\mathcal{A}}} - \alpha \mathrm{I}$ is not $\mathcal{A}$-Fredholm $\rbrace$.
	Then we have the following proposition. \\
	\begin{proposition} \label{p310}
		For given $\mathrm{F},\mathrm{C},\mathrm{D} \in B^{a}  (H_{\mathcal{A}})$, one has $$\sigma_{e}^{\mathcal{A}}({\mathbf{M}_{\mathrm{C}}^{\mathcal{A}}}) \subset (\sigma_{e}^{\mathcal{A}}(\mathrm{F}) \cup \sigma_{e}^{\mathcal{A}}(\mathrm{D})) .$$	
	\end{proposition}
	\begin{proof}
		Observe first that 
		\begin{center}
			${\mathbf{M}_{\mathrm{C}}^{\mathcal{A}}} - \alpha \mathrm{I}=
			\left\lbrack
			\begin{array}{ll}
			1 & 0 \\
			0 & \mathrm{D}-\alpha 1 \\
			\end{array}
			\right \rbrack
			$
			$
			\left\lbrack
			\begin{array}{ll}
			1 & \mathrm{C} \\
			0 & 1 \\
			\end{array}
			\right \rbrack
			$
			$
			\left\lbrack
			\begin{array}{ll}
			\mathrm{F}- \alpha 1 & 0 \\
			0 & 1 \\
			\end{array}
			\right \rbrack .
			$	  
		\end{center} 
		Now 
		$
		\left\lbrack
		\begin{array}{ll}
		1 & \mathrm{C} \\
		0 & 1 \\
		\end{array}
		\right \rbrack
		$ 
		is clearly invertible in $ B^{a}(H_{\mathcal{A}} \oplus H_{\mathcal{A}}) $ with inverse
		$
		\left\lbrack
		\begin{array}{ll}
		1 & -\mathrm{C} \\
		0 & 1 \\
		\end{array}
		\right \rbrack
		,$ 
		so it follows that
		$
		\left\lbrack
		\begin{array}{ll}
		1 & \mathrm{C} \\
		0 & 1 \\
		\end{array}
		\right \rbrack
		$  
		is $\mathcal{A}$-Fredholm. If, in addition both
		$\left\lbrack
		\begin{array}{ll}
		\mathrm{F}- \alpha 1 & 0 \\
		0 & 1 \\
		\end{array}
		\right \rbrack 
		$
		and
		$\left\lbrack
		\begin{array}{ll}
		1 & 0 \\
		0 & \mathrm{D}-\alpha 1 \\
		\end{array}
		\right \rbrack
		$
		are $\mathcal{A}$-Fredholm, then ${\mathbf{M}_{\mathrm{C}}^{\mathcal{A}}} - \alpha \mathrm{I}$ is $\mathcal{A}$-Fredholm being a composition of  $\mathcal{A}$-Fredholm operators. But, if $\mathrm{F} - \alpha \mathrm{I}$ is $\mathcal{A}$-Fredholm, then clearly 
		$ \left\lbrack
		\begin{array}{ll}
		\mathrm{F}- \alpha 1 & 0 \\
		0 & 1 \\
		\end{array}
		\right \rbrack  $ is $\mathcal{A}$-Fredholm, and similarly if $ \mathrm{D}- \alpha \mathrm{I}  $ is $\mathcal{A}$-Fredholm, then\\
		$ 
		\left\lbrack
		\begin{array}{ll}
		1 & 0 \\
		0 & \mathrm{D}-\alpha 1 \\
		\end{array}
		\right \rbrack  
		$
		is $\mathcal{A}$-Fredholm. Thus, if both $ \mathrm{F}-\alpha \mathrm{I}   $ and $ \mathrm{D}-\alpha \mathrm{I}  $ are $\mathcal{A}$-Fredholm, then $ {\mathbf{M}_{\mathrm{C}}^{\mathcal{A}}} - \alpha \mathrm{I}  $ is $\mathcal{A}$-Fredholm. The proposition follows.	
	\end{proof}
	%
	%
	This proposition just gives an inclusion. We are going to investigate in which cases the equality holds. To this end we introduce first the following theorem.
	\begin{theorem} \label{t320} 
		Let $ \mathrm{F},\mathrm{D} \in B^{a}  (H_{\mathcal{A}}) .$ If $ {\mathbf{M}_{\mathrm{C}}^{\mathcal{A}}} \in \mathcal{M} \Phi (H_{\mathcal{A}} \oplus H_{\mathcal{A}})   $ for some\\
		$\mathrm{C}\in B^{a}(H_{\mathcal{A}}),$
		then $ \mathrm{F} \in \mathcal{M} \Phi_{+} (H_{\mathcal{A}} ),  \mathrm{D} \in \mathcal{M} \Phi_{-} (H_{\mathcal{A}} )  $ and for all decompositions
		$$H_{\mathcal{A}} = M_{1} \tilde \oplus {N_{1}}\stackrel{\mathrm{F}}{\longrightarrow} M_{2} \tilde \oplus N_{2}= H_{\mathcal{A}} ,$$ 
		$$H_{\mathcal{A}} = M_{1}^{\prime} \tilde \oplus {N_{1}^{\prime}}\stackrel{\mathrm{D}}{\longrightarrow} M_{2}^{\prime} \tilde \oplus N_{2}^{\prime}= H_{\mathcal{A}} $$
		w.r.t. which $\mathrm{F}, \mathrm{D}$ have matrices 
		$\left\lbrack
		\begin{array}{ll}
		\mathrm{F}_{1} & 0 \\
		0 & \mathrm{F}_{4} \\
		\end{array}
		\right \rbrack
		,$  
		$\left\lbrack
		\begin{array}{ll}
		\mathrm{D}_{1} & 0 \\
		0 & \mathrm{D}_{4} \\
		\end{array}
		\right \rbrack  ,$
		respectively, where $ \mathrm{F}_{1},\mathrm{D}_{1} $ are isomorphisms, and $N_{1},N_{2}^{\prime}   $ are finitely generated, there exist closed submodules\\
		$ \tilde {N_{1}^{\prime}}, \tilde{\tilde{N_{1}^{\prime}}}, \tilde {N_{2}},\tilde{\tilde{N_{2}}}  $ such that  $N_{2}  \cong \tilde{N_{2}}, N_{1}^{\prime} \cong \tilde{N_{1}^{\prime}} $, $ \tilde{\tilde{N_{2}}}  $ and $  \tilde{\tilde{N_{1}^{\prime}}}  $ are finitely generated  and 
		$$\tilde {N_{2}} \tilde{\oplus} \tilde{\tilde{N_{2}}}  \cong \tilde {N_{1}^{\prime}} \tilde{\oplus} \tilde{\tilde{N_{1}^{\prime}}}.$$ 
	\end{theorem}
	\begin{proof}
		Again write $ {\mathbf{M}_{\mathrm{C}}^{\mathcal{A}}} $ as  ${\mathbf{M}_{\mathrm{C}}^{\mathcal{A}}}=\mathrm{D}^\prime  \mathrm{C}^\prime  \mathrm{F}^\prime $ where 
		$$\mathrm{F}^{\prime}=
		\left\lbrack
		\begin{array}{ll}
		\mathrm{F} & 0 \\
		0 & 1 \\
		\end{array}
		\right \rbrack ,
		\mathrm{C}^{\prime}=\left\lbrack
		\begin{array}{ll}
		1 & \mathrm{C} \\
		0 & 1 \\
		\end{array}
		\right \rbrack ,
		\mathrm{D}^{\prime}=\left\lbrack
		\begin{array}{ll}
		1 & 0 \\
		0 & \mathrm{D} \\
		\end{array}
		\right \rbrack
		.$$
		Since $ {\mathbf{M}_{\mathrm{C}}^{\mathcal{A}}}  $ is $\mathcal{A}$-Fredholm, if 
		$$H_{\mathcal{A}} {\oplus} H_{\mathcal{A}} = M \tilde \oplus N\stackrel{{\mathbf{M}_{\mathrm{C}}^{\mathcal{A}}}}{\longrightarrow} M^{\prime} \tilde \oplus N^{\prime}= H_{\mathcal{A}} {\oplus} H_{\mathcal{A}}    $$ 
		is a decomposition w.r.t. which $ {\mathbf{M}_{\mathrm{C}}^{\mathcal{A}}}  $ has the matrix 
		$\left\lbrack
		\begin{array}{ll}
		({\mathbf{M}_{\mathrm{C}}^{\mathcal{A}}})_{1} & 0 \\
		0 & ({\mathbf{M}_{\mathrm{C}}^{\mathcal{A}}})_{4}  \\
		\end{array}
		\right \rbrack   $
		where $ ({\mathbf{M}_{\mathrm{C}}^{\mathcal{A}}})_{1}  $ is an isomorphism and $ N, N^{\prime}   $ are finitely generated, then  by Lemma 2.2 and also using that $ \mathrm{C}^{\prime}  $ is invertible, one may easily deduce that there exists a chain of decompositions
		$$H_{\mathcal{A}} {\oplus} H_{\mathcal{A}} = M \tilde \oplus N\stackrel{\mathrm{F}^{\prime}}{\longrightarrow}  R_{1} \tilde{\oplus} {R_{2}}\stackrel{\mathrm{C}^{\prime}}{\longrightarrow}   \mathrm{C}^{\prime}(R_{1}) \tilde{\oplus} \mathrm{C}^{\prime}(R_{2})\stackrel{\mathrm{D}^{\prime}}{\longrightarrow}  M^{\prime} \tilde \oplus N^{\prime}= H_{\mathcal{A}} {\oplus} H_{\mathcal{A}}$$ 
		w.r.t. which $\mathrm{F}^{\prime}, \mathrm{C}^{\prime}, \mathrm{D}^{\prime}   $ have matrices
		$$\left\lbrack
		\begin{array}{ll}
		\mathrm{F}_{1}^{\prime} & 0 \\
		0 & \mathrm{F}_{4}^{\prime} \\
		\end{array}
		\right \rbrack ,
		\left\lbrack
		\begin{array}{ll}
		\mathrm{C}_{1}^{\prime} & 0 \\
		0 & \mathrm{C}_{4}^{\prime} \\
		\end{array}
		\right \rbrack ,
		\left\lbrack
		\begin{array}{ll}
		\mathrm{D}_{1}^{\prime} & \mathrm{D}_{2}^{\prime} \\
		0 & \mathrm{D}_{4}^{\prime} \\
		\end{array}
		\right \rbrack ,$$
		respectively, 
		where $ \mathrm{F}_{1}^{\prime} ,\mathrm{C}_{1}^{\prime}, \mathrm{C}_{4}^{\prime}, \mathrm{D}_{1}^{\prime}     $ are isomorphisms. So $\mathrm{D}^{\prime}$ has the matrix\\
		$\left\lbrack
		\begin{array}{ll}
		\mathrm{D}_{1}^{\prime} & 0 \\
		0 & \mathrm{D}_{4}^{\prime} \\
		\end{array}
		\right \rbrack
		$ 
		w.r.t. the decomposition 
		$$H_{\mathcal{A}} {\oplus} H_{\mathcal{A}}=\mathrm{W}\mathrm{C}^{\prime}(R_{1}) \tilde{\oplus} \mathrm{W}\mathrm{C}^{\prime}(R_{2})\stackrel{\mathrm{D}^{\prime}}{\longrightarrow} M^{\prime} \tilde{\oplus}  N^{\prime} =  H_{\mathcal{A}} {\oplus} H_{\mathcal{A}},$$
		where $\mathrm{W}$ has the matrix
		$\left\lbrack
		\begin{array}{ll}
		1 & -{\mathrm{D}_{1}^{\prime}}^{-1}\mathrm{D}_{2}^{\prime} \\
		0 & 1 \\
		\end{array}
		\right \rbrack
		$ 
		w.r.t the decomposition 
		$$\mathrm{C}^{\prime}(R_{1}) \tilde{\oplus} \mathrm{C}^{\prime}(R_{2})\stackrel{\mathrm{W}}{\longrightarrow} \mathrm{C}^{\prime}(R_{1}) \tilde{\oplus} \mathrm{C}^{\prime}(R_{2})$$
		and is therefore an isomorphism.
		\\
		It follows from this that 
		$$\mathrm{F}^{\prime} \in \mathcal{M} \Phi_{+}  (H_{\mathcal{A}} {\oplus} H_{\mathcal{A}}), \mathrm{D}^{\prime} \in \mathcal{M} \Phi_{-}  (H_{\mathcal{A}} {\oplus} H_{\mathcal{A}}) ,$$
		as $N$ and $N^{\prime}   $ are finitely generated submodules of $H_{\mathcal{A}} {\oplus} H_{\mathcal{A}}$ . Moreover \\
		$ R_{2} \cong \mathrm{W}\mathrm{C}^{\prime} (R_{2}),$  as $ \mathrm{W} \mathrm{C}^{\prime} $ is an isomorphism.	
		
		%
		Since there exists an adjointable isomorphism between $H_{\mathcal{A}}$ and $H_{\mathcal{A}} \oplus H_{\mathcal{A}},$ using \cite[Theorem 2.2 ]{I} and \cite[ Theorem 2.3]{I}  it is easy to deduce that $\mathrm{F}^{\prime}$ is left invertible and $\mathrm{D}^{\prime}$ is right invertible in the „Calkin“ algebra on $B^{a}(H_{\mathcal{A}} \oplus H_{\mathcal{A}}) /$ $ K(H_{\mathcal{A}}\oplus H_{\mathcal{A}}).$ It follows from this that $\mathrm{F}$ is left invertible and $\mathrm{D}$ is right invertible in the „Calkin“ algebra $B^{a}(H_{\mathcal{A}}) / K(H_{\mathcal{A}})   ,$ hence $\mathrm{F} \in \mathcal{M} \Phi_{+}(H_{\mathcal{A}})$ and $\mathrm{D} \in \mathcal{M} \Phi_{-}(H_{\mathcal{A}})$ again by \cite[Theorem 2.2 ]{I} and  \cite[Theorem 2.3 ]{I}, respectively. Choose arbitrary $ \mathcal{M} \Phi_{+}$ and $\mathcal{M} \Phi_{-}$ decompositions for $\mathrm{F}$ and $\mathrm{D}$ respectively i.e.
		$$ H_{\mathcal{A}} = M_{1} \tilde \oplus {N_{1}}\stackrel{\mathrm{F}}{\longrightarrow} M_{2} \tilde \oplus N_{2}= H_{\mathcal{A}} ,$$  
		$$H_{\mathcal{A}} = M_{1}^{\prime} \tilde \oplus {N_{1}^{\prime}}\stackrel{\mathrm{D}}{\longrightarrow} M_{2}^{\prime} \tilde \oplus N_{2}^{\prime}= H_{\mathcal{A}}. $$
		Then 
		$$ H_{\mathcal{A}} \oplus H_{\mathcal{A}} =(M_{1} \oplus H_{\mathcal{A}}) \tilde{\oplus} ( N_{1} \oplus \lbrace 0 \rbrace )$$ 
		$$\downarrow \mathrm{F}^{\prime}$$
		$$ H_{\mathcal{A}} \oplus H_{\mathcal{A}} =(M_{2} \oplus H_{\mathcal{A}}) \tilde{\oplus} ( N_{2} \oplus \lbrace 0 \rbrace )  $$ 
		and
		$$ H_{\mathcal{A}} \oplus H_{\mathcal{A}} =(H_{\mathcal{A}}  \oplus M_{1}^{\prime}) \tilde{\oplus} (  \lbrace 0 \rbrace \oplus  N_{1}^{\prime})  $$ 
		$$\downarrow \mathrm{D}^{\prime}$$
		$$ H_{\mathcal{A}} \oplus H_{\mathcal{A}} =(H_{\mathcal{A}} \oplus M_{2}^{\prime} ) \tilde{\oplus} (  \lbrace 0 \rbrace \oplus  N_{2}^{\prime})   $$ 
		are $\mathcal{M} \Phi_{+}$ and $\mathcal{M} \Phi_{-}$ decompositions for $\mathrm{F}^{\prime}$ and $\mathrm{D}^{\prime}$ respectively.
		Hence the decomposition
		$$H_{\mathcal{A}} \oplus H_{\mathcal{A}} = M  \tilde{\oplus} N\stackrel{\mathrm{F}^{\prime}}{\longrightarrow} R_{1} \tilde{\oplus} R_{2} = H_{\mathcal{A}} \oplus H_{\mathcal{A}}  $$
		and the $\mathcal{M} \Phi_{+}$ decomposition given above for $\mathrm{F}^{\prime}$ are two $\mathcal{M} \Phi_{+}$ decompositions for $\mathrm{F}^{\prime}.$ Again, since there exists an adjointable  isomorphism between $H_{\mathcal{A}} \oplus H_{\mathcal{A}}$ and $ H_{\mathcal{A}} ,$ we may apply  \cite[Corollary 2.18 ]{I} to operator $\mathrm{F}^{\prime}$ to deduce that\\
		$(( N_{2} \oplus \lbrace 0 \rbrace ) \tilde{\oplus} P )\cong (R_{2} \tilde{\oplus} \tilde{P})$ for some finitely generated submodules $P, \tilde{P}$ of $ H_{\mathcal{A}} \oplus H_{\mathcal{A}}.$ Similarly, since
		$$ H_{\mathcal{A}} \oplus H_{\mathcal{A}} = \mathrm{W}\mathrm{C}^{\prime}(R_{1}) \tilde \oplus \mathrm{W}\mathrm{C}^{\prime} (R_{2})\stackrel{\mathrm{D}^{\prime}}{\longrightarrow}  M^{\prime} \tilde \oplus N^{\prime}= H_{\mathcal{A}} \oplus H_{\mathcal{A}}$$
		and
		$$ H_{\mathcal{A}} \oplus H_{\mathcal{A}} =(H_{\mathcal{A}}  \oplus M_{1}^{\prime}) \tilde{\oplus} (  \lbrace 0 \rbrace \oplus  N_{1}^{\prime})  $$ 
		$$\downarrow \mathrm{D}^{\prime}$$
		$$ H_{\mathcal{A}} \oplus H_{\mathcal{A}} =(H_{\mathcal{A}}  \oplus M_{2}^{\prime}) \tilde{\oplus} (  \lbrace 0 \rbrace \oplus  N_{2}^{\prime})  $$
		are two $\mathcal{M} \Phi_{-}$ decompositions for $\mathrm{D}^{\prime},$ we may by the same arguments apply\\
		\cite[Corollary 2.19 ]{I} to the operator $\mathrm{D}^{\prime}$ to deduce that 
		$$((  \lbrace 0 \rbrace \oplus  N_{1}^{\prime}) \tilde{\oplus} P^{\prime}) \cong  (\mathrm{W}\mathrm{C}^{\prime}(R_{2}) \tilde{\oplus} \tilde{P}^{\prime}) $$ 
		for some finitely generated submodules $P^{\prime} , \tilde{P}^{\prime} $ of $ H_{\mathcal{A}} \oplus H_{\mathcal{A}}.$ Since $\mathrm{W}\mathrm{C}^{\prime}$ is an isomorphism, we get 
		$$  ((\mathrm{W}\mathrm{C}^{\prime}(R_{2})  \tilde{\oplus} \tilde{P}^{\prime} ) \oplus   \tilde{P}) \cong         (\mathrm{W}\mathrm{C}^{\prime}(R_{2}) \oplus   \tilde{P}^{\prime} \oplus  \tilde{P}) \cong (R_{2}    \oplus  \tilde{P}  \oplus \tilde{P}^{\prime} )\cong   ((R_{2}    \tilde{\oplus} \tilde{P} )  \oplus  \tilde{P}^{\prime})  .$$ 
		Hence 
		$$ ((( N_{2} \oplus \lbrace 0 \rbrace ) \tilde{\oplus} P) \oplus \tilde{P}^{\prime}) \cong (( (\lbrace 0 \rbrace  \oplus N_{1}^{\prime})  \tilde{\oplus} {P}^{\prime} ) \oplus \tilde{P}) .$$
		This gives $ (N_{2} \oplus P \tilde{\oplus}  \tilde{P}^{\prime}) \cong (N_{1}^{\prime} \oplus P^{\prime} \oplus  \tilde{P}) $ (Here $\oplus$ always denotes the direct sum of modules in the sense of  \cite[Example 1.3.4 ]{MT}). Now
		$$N_{2} \oplus P \tilde{\oplus}  \tilde{P}^{\prime}=(N_{2}  \oplus  \lbrace 0 \rbrace \oplus  \lbrace 0 \rbrace   ) \tilde{\oplus} (\lbrace 0 \rbrace  \oplus  P \oplus {P}^{\prime}),$$
		$$N_{1} \oplus P^{\prime} \tilde{\oplus}  {P}^{\prime}=(N_{1}^{\prime}  \oplus  \lbrace 0 \rbrace \oplus  \lbrace 0 \rbrace   ) \tilde{\oplus} (\lbrace 0 \rbrace  \oplus  P^{\prime} \oplus \tilde{P})$$
		and they are submodules of $L_{5}(H_{\mathcal{A}})$ which is isomorphic to $H_{\mathcal{A}}$ ( the notation $L_{5}(H_{\mathcal{A}})$ is as in \cite[Example 1.3.4 ]{MT}). Call the isomorphism betwen $H_{\mathcal{A}} $ for and  $L_{5}(H_{\mathcal{A}})$ for $\mathrm{U}$ and set
		$$\tilde{N_{2}}=\mathrm{U}(N_{2}  \oplus  \lbrace 0 \rbrace \oplus  \lbrace 0 \rbrace   ), \tilde {\tilde{ N_{2} }} =\mathrm{U}(\lbrace 0 \rbrace  \oplus  P \oplus P^{\prime}) ,$$
		$$\tilde{N_{1}}=\mathrm{U}(N_{1}^{\prime}  \oplus  \lbrace 0 \rbrace \oplus  \lbrace 0 \rbrace   ), \tilde {\tilde{ N_{1} }}^{\prime} =\mathrm{U}(\lbrace 0 \rbrace  \oplus  P^{\prime} \oplus \tilde{P}) .$$
		Since $ P,P^{\prime},\tilde{ P },\tilde{ P }^{\prime}  $ are finitely generated, the result follows. 
	\end{proof}
	%
	\begin{remark}
		\cite[Theorem 3.2 ]{DDj}, part $(1)\Rightarrow (2)$ follows actually as a corollary from our Theorem \ref{t320} in the case when $X=Y=H,$ where $H$ is a Hilbert space. Indeed, by Theorem \ref{t320} if ${\mathbf{M}_{\mathrm{C}}} \in \Phi(H \oplus H),$ then $\mathrm{F} \in \Phi_{+}(H)$ and $\mathrm{D} \in \Phi_{-}(H)$. Hence $\mathrm{Im}\mathrm{F}$ and $\mathrm{Im}\mathrm{D}$ are closed, $\dim \ker \mathrm{F}, \dim \mathrm{Im}\mathrm{D}^{\bot}< \infty.$ W.r.t. the decompositions $H= \ker \mathrm{F}^{\bot} \oplus \ker \mathrm{F}\stackrel{\mathrm{F}}{\longrightarrow} \mathrm{Im}\mathrm{F} \oplus \mathrm{Im}\mathrm{F}^{\bot}=H$ and $$H= \ker \mathrm{D}^{\bot} \oplus \ker \mathrm{D}\stackrel{\mathrm{D}}{\longrightarrow} \mathrm{Im}\mathrm{D} \oplus \mathrm{Im}\mathrm{D}^{\bot}=H,$$ 
		$\mathrm{F},\mathrm{D}$ have matrices
		$\left\lbrack
		\begin{array}{ll}
		\mathrm{F}_{1} & 0 \\
		0 & \mathrm{F}_{4}\\
		\end{array}
		\right \rbrack,
		$
		$\left\lbrack
		\begin{array}{ll}
		\mathrm{D}_{1} & 0 \\
		0 & \mathrm{D}_{4} \\
		\end{array}
		\right \rbrack,
		$
		respectively, where $\mathrm{F}_{1}, \mathrm{D}_{1}$ are isomorphisms.\\
		From Theorem \ref{t320} it follows that there exist closed subspaces $\tilde{N_{2}}, \tilde{\tilde{N_{2}}},\tilde{N_{1}^{\prime}},\tilde{\tilde{N_{1}^{\prime}}} $ such that $\tilde{N_{2}}\cong \mathrm{Im} \mathrm{F}^{\bot} , \tilde{N_{1}^{\prime}} \cong \ker \mathrm{D}, \dim \tilde{\tilde{N_{2}}},\dim  \tilde{\tilde{N_{1}^{\prime}}} < \infty$ and \\
		$(\tilde{N_{2}} \tilde{\oplus}  \tilde{\tilde{N_{2}}}) \cong (\tilde{\tilde{N_{1}^{\prime}}} \tilde{\oplus} \tilde{\tilde{N_{1}^{\prime}}}) $. But this just means that $\mathrm{Im}\mathrm{F}^{\bot}$ and $\ker \mathrm{D}$ are isomorphic up to a finite dimensional subspace in the sense of \cite[Definition 2.2 ]{DDj} because we consider Hilbert subspaces now.	
	\end{remark}
	%
	\begin{proposition} \label{p340}
		Suppose that there exists some $\mathrm{C} \in  B^{a}  (H_{\mathcal{A}})   $ such that the inclusion $\sigma_{e}^{\mathcal{A}}({\mathbf{M}_{\mathrm{C}}^{\mathcal{A}}}) \subset \sigma_{e}^{\mathcal{A}}(\mathrm{F}) \cup \sigma_{e}^{\mathcal{A}}(\mathrm{D})   $ is proper. Then for any 
		$$\alpha \in [\sigma_{e}^{\mathcal{A}}(\mathrm{F}) \cup \sigma_{e}^{\mathcal{A}}(\mathrm{D}) ] \setminus  \sigma_{e}^{\mathcal{A}}({\mathbf{M}_{\mathrm{C}}^{\mathcal{A}}})  $$ 
		we have 
		$$\alpha \in \sigma_{e}^{\mathcal{A}}(\mathrm{F}) \cap \sigma_{e}^{\mathcal{A}}(\mathrm{D}).$$
	\end{proposition}
	\begin{proof}
		Assume that $$ \alpha \in [\sigma_{e}^{\mathcal{A}}(\mathrm{F})  \setminus \sigma_{e}^{\mathcal{A}}(\mathrm{D}) ]  \setminus  \sigma_{e}^{\mathcal{A}}({\mathbf{M}_{\mathrm{C}}^{\mathcal{A}}})  .$$
		Then $(\mathrm{F}- \alpha 1) \notin \mathcal{M} \Phi (H_{\mathcal{A}})   $ and $(\mathrm{D}- \alpha 1) \in \mathcal{M} \Phi (H_{\mathcal{A}})      .$ Moreover, since\\
		$\alpha \notin \sigma_{e}^{\mathcal{A}}({\mathbf{M}_{\mathrm{C}}^{\mathcal{A}}})    $, then $({\mathbf{M}_{\mathrm{C}}^{\mathcal{A}}}- \alpha 1)$ is $\mathcal{A}$-Fredholm. From Theorem \ref{t320}, it follows that $(\mathrm{F}- \alpha 1) \in \mathcal{M} \Phi_{+} (H_{\mathcal{A}})    .$ Since $(\mathrm{F}- \alpha 1) \in \mathcal{M} \Phi_{+} (H_{\mathcal{A}}),(\mathrm{D}- \alpha 1) \in \mathcal{M} \Phi (H_{\mathcal{A}})     $, we can find decompositions 
		$$H_{\mathcal{A}} = M_{1} \tilde \oplus {N_{1}}\stackrel{\mathrm{F}-\alpha 1}{\longrightarrow} M_{2} \tilde \oplus N_{2}= H_{\mathcal{A}} ,$$
		$$H_{\mathcal{A}} = M_{1}^{\prime} \tilde \oplus {N_{1}^{\prime}}\stackrel{\mathrm{D}-\alpha 1}{\longrightarrow} M_{2}^{\prime} \tilde \oplus N_{2}^{\prime}= H_{\mathcal{A}} $$
		w.r.t. which $\mathrm{F}- \alpha 1,\mathrm{D}- \alpha 1   $ have matrices
		\begin{center}
			$
			\left\lbrack
			\begin{array}{ll}
			(\mathrm{F}- \alpha 1)_{1} & 0 \\
			0 & (\mathrm{F}- \alpha 1)_{4} \\
			\end{array}
			\right \rbrack ,
			$
			$
			\left\lbrack
			\begin{array}{ll}
			(\mathrm{D}- \alpha 1)_{1} & 0 \\
			0 & (\mathrm{D}- \alpha 1)_{4} \\
			\end{array}
			\right \rbrack ,
			$
		\end{center}
		respectively, where $(\mathrm{F}- \alpha 1)_{1} , (\mathrm{D}- \alpha 1)_{1}    $ are isomorphisms, $ N_{1},N_{1}^{\prime}  $ and $N_{2}^{\prime} $ are finitely generated. By Theorem \ref{t320} there exist then closed submodules\\
		$\tilde{N_{2}} ,\tilde{\tilde{N_{2}}},\tilde{N_{1}^{\prime}} , \tilde{\tilde{N_{1}^{\prime}}}  $ such that $N_{2} \cong \tilde{N_{2}}, N_{1}^{\prime} \cong \tilde{N_{1}^{\prime}},$ $ (\tilde{N_{2}} \tilde{\oplus} \tilde{\tilde{N_{2}}})  \cong (\tilde{N_{1}^{\prime}} \tilde{\oplus} \tilde{\tilde{N_{1}^{\prime}}})$ and $\tilde{\tilde{N_{2}}} , \tilde{\tilde{N_{1}^{\prime}}}   $ are finitely generated. But then, since $N_{1}^{\prime}  $ is finitely generated (as $(\mathrm{D}- \alpha 1) \in \mathcal{M} \Phi (H_{\mathcal{A}})  $), we get that $\tilde{N_{1}^{\prime}}   $ is finitely generated being isomorphic to ${N_{1}^{\prime}}.$ Hence $ ({\tilde{N_{1}^{\prime}}}  \tilde{\oplus} \tilde{\tilde{N_{1}^{\prime}}})  $ is finitely generated also (as both $\tilde{N_{1}^{\prime}}   $ and $\tilde{\tilde{N_{1}^{\prime}}}   $ are finitely generated). Thus $ ({\tilde{N_{2}}}  \tilde{\oplus} \tilde{\tilde{N_{2}}})   $ is finitely generated as well, so ${\tilde{N_{2}}}   $ is finitely generated.
		Therefore $N_{2} $ is finitely generated, being isomorphic to ${\tilde{N_{2}}} .$ Hence $\mathrm{F}-\alpha 1   $ is in $\mathcal{M} \Phi (H_{\mathcal{A}}) .$ This contradicts the choice of 
		$$ \alpha \in [\sigma_{e}^{\mathcal{A}}(\mathrm{F})  \setminus \sigma_{e}^{\mathcal{A}}(\mathrm{D}) ]  \setminus  \sigma_{e}^{\mathcal{A}}({\mathbf{M}_{\mathrm{C}}^{\mathcal{A}}})  .$$
		Thus 
		$$[\sigma_{e}^{\mathcal{A}}(\mathrm{F})  \setminus \sigma_{e}^{\mathcal{A}}(\mathrm{D}) ]  \setminus  \sigma_{e}^{\mathcal{A}}({\mathbf{M}_{\mathrm{C}}^{\mathcal{A}}}) = \varnothing .   $$ Analogously we can prove
		$$[\sigma_{e}^{\mathcal{A}}(\mathrm{D})  \setminus \sigma_{e}^{\mathcal{A}}(\mathrm{F}) ]  \setminus  \sigma_{e}^{\mathcal{A}}({\mathbf{M}_{\mathrm{C}}^{\mathcal{A}}}) = \varnothing .   $$ 
		The proposition follows.	
	\end{proof}
	Next, we define the following classes of operators on $H_{\mathcal{A}}:   $ 
	$$\mathcal{M} S_{+} (H_{\mathcal{A}}) = \lbrace \mathrm{F} \in B^{a}(H_{\mathcal{A}} \mid (\mathrm{F}-\alpha 1) \in \mathcal{M} \Phi_{-}^{+} (H_{\mathcal{A}}) $$ 
	$$ \mbox{ whenever } \alpha \in  Z(\mathcal{A})   \mbox{ and } (\mathrm{F}-\alpha 1) \in \mathcal{M} \Phi_{\pm } (H_{\mathcal{A}}) \rbrace    ,$$
	$$\mathcal{M} S_{-} (H_{\mathcal{A}}) = \lbrace \mathrm{F} \in B^{a}(H_{\mathcal{A}} \mid (\mathrm{F}-\alpha 1) \in \mathcal{M} \Phi_{+}^{-} (H_{\mathcal{A}}) $$ 
	$$ \mbox{ whenever } \alpha \in  Z(\mathcal{A})   \mbox{ and } (\mathrm{F}-\alpha 1) \in \mathcal{M} \Phi_{\pm } (H_{\mathcal{A}}) \rbrace    .$$
	\begin{proposition} \label{p350} 
		If $ \mathrm{F} \in \mathcal{M} S_{+} (H_{\mathcal{A}})$ or $\mathrm{D} \in  \mathcal{M} S_{-} (H_{\mathcal{A}}) $, then for all\\
		$\mathrm{C} \in B^{a}  (H_{\mathcal{A}})   $, we have $$\sigma_{e}^{\mathcal{A}}({\mathbf{M}_{\mathrm{C}}^{\mathcal{A}}}) =    \sigma_{e}^{\mathcal{A}}(\mathrm{F})  \cup \sigma_{e}^{\mathcal{A}}(\mathrm{D})$$	
	\end{proposition}
	\begin{proof}
		By Proposition \ref{p340}, it suffices to show the inclusion. Assume that
		$$\alpha \in   [\sigma_{e}^{\mathcal{A}}(\mathrm{F})  \cup \sigma_{e}^{\mathcal{A}}(\mathrm{D}) ]  \setminus  \sigma_{e}^{\mathcal{A}}({\mathbf{M}_{\mathrm{C}}^{\mathcal{A}}}). $$ 
		Then, $ ({\mathbf{M}_{\mathrm{C}}^{\mathcal{A}}} - \alpha 1) \in \mathcal{M} \Phi (H_{\mathcal{A}} \oplus H_{\mathcal{A}})   .$ By Theorem \ref{t320}, we have
		$$(\mathrm{F} - \alpha 1) \in \mathcal{M} \Phi_{+} (H_{\mathcal{A}}),(\mathrm{D} - \alpha 1) \in \mathcal{M} \Phi_{-} (H_{\mathcal{A}})    .$$
		Let again
		$$H_{\mathcal{A}} = M_{1} \tilde \oplus {N_{1}}\stackrel{\mathrm{F}-\alpha 1}{\longrightarrow} M_{2} \tilde \oplus N_{2}= H_{\mathcal{A}} ,$$
		$$H_{\mathcal{A}} = M_{1}^{\prime} \tilde \oplus {N_{1}^{\prime}}\stackrel{\mathrm{D}-\alpha 1}{\longrightarrow} M_{2}^{\prime} \tilde \oplus N_{2}^{\prime}= H_{\mathcal{A}} $$ 
		be decompositions w.r.t. which $\mathrm{F} - \alpha 1,\mathrm{D} - \alpha 1    $ have matrices
		\begin{center}
			$
			\left\lbrack
			\begin{array}{ll}
			(\mathrm{F}- \alpha 1)_{1} & 0 \\
			0 & (\mathrm{F}- \alpha 1)_{4} \\
			\end{array}
			\right \rbrack 
			$,
			$
			\left\lbrack
			\begin{array}{ll}
			(\mathrm{D}- \alpha 1)_{1} & 0 \\
			0 & (\mathrm{D}- \alpha 1)_{4} \\
			\end{array}
			\right \rbrack ,
			$	
		\end{center}
		respectively, where $(\mathrm{F}- \alpha 1)_{1}, (\mathrm{D}- \alpha 1)_{1}   $, are isomorphisms and $N_{1},N_{2}^{\prime}   $ are finitely generated submodules of $H_{\mathcal{A}}   .$ Again, by Theorem \ref{t320}, there exist closed submodules $\tilde{N_{2}} ,\tilde{\tilde{N_{2}}},\tilde{N_{1}^{\prime}} , \tilde{\tilde{N_{1}^{\prime}}}  $ such that $N_{2} \cong \tilde{N_{2}}, N_{1}^{\prime} \cong \tilde{N_{1}^{\prime}},$ $ (\tilde{N_{2}} \tilde{\oplus} \tilde{\tilde{N_{2}}})  \cong (\tilde{N_{1}^{\prime}} \tilde{\oplus} \tilde{\tilde{N_{1}^{\prime}}})$ and $\tilde{\tilde{N_{2}}} , \tilde{\tilde{N_{1}^{\prime}}}   $ are finitely generated submodules. If $\mathrm{F} \in \mathcal{M} S_{+} (H_{\mathcal{A}})   $, then since \\
		$(\mathrm{F}- \alpha 1) \in \mathcal{M} \Phi_{\pm} (H_{\mathcal{A}})  $, we get that $(\mathrm{F}- \alpha 1) \in \mathcal{M} \Phi_{-}^{+} (H_{\mathcal{A}})    $. Thus\\
		$(\mathrm{F}- \alpha 1) \in \mathcal{M} \Phi_{-} (H_{\mathcal{A}})    $ in particular. So $(\mathrm{F}- \alpha 1) \in \mathcal{M} \Phi_{+} (H_{\mathcal{A}}) \cap \mathcal{M} \Phi_{-} (H_{\mathcal{A}})   $ and by \cite[Corollary 2.4]{I}, we know that $ \mathcal{M} \Phi_{+} (H_{\mathcal{A}}) \cap  \mathcal{M} \Phi_{-} (H_{\mathcal{A}})= \mathcal{M} \Phi (H_{\mathcal{A}}).$ Then, by \cite[Lemma 2.16]{I}, we have that $ N_{2}  $ must be finitely generated, hence $ \tilde{N_{2}}  $ must be finitely generated. Thus $ {\tilde{N_{2}}}  \tilde{\oplus} \tilde{\tilde{N_{2}}}   $ is finitely generated. \\
		Since $ ({\tilde{N_{2}}}  \tilde{\oplus} \tilde{\tilde{N_{2}}}) \cong   ({\tilde{N_{1}^{\prime}}}  \tilde{\oplus} \tilde{\tilde{N_{1}^{\prime}}}) ,$ it follows that $ \tilde{N_{1}^{\prime}}  $ is finitely generated, hence $N_{1}^{\prime}$ is finitely generated also. So $(\mathrm{D}- \alpha 1) \in \mathcal{M} \Phi (H_{\mathcal{A}})   $. Similarly, we can show that if $\mathrm{D} \in S_{-}(H_{\mathcal{A}})   $, then $(\mathrm{F}- \alpha 1) \in \mathcal{M} \Phi (H_{\mathcal{A}})     $. In both cases $(\mathrm{F}- \alpha 1) \in \mathcal{M} \Phi (H_{\mathcal{A}})    $ and $(\mathrm{D}- \alpha 1) \in \mathcal{M} \Phi (H_{\mathcal{A}})    ,$ which contradicts that $\alpha \in   \sigma_{e}^{\mathcal{A}}(\mathrm{F})  \cup \sigma_{e}^{\mathcal{A}}(\mathrm{D})   .$	
	\end{proof}
	\begin{theorem} \label{t360}
		Let $ \mathrm{F} \in \mathcal{M} \Phi_{+} (H_{\mathcal{A}}), \mathrm{D} \in \mathcal{M} \Phi_{-} (H_{\mathcal{A}})  $ and suppose that there exist decompositions
		$$H_{\mathcal{A}} = M_{1} \tilde \oplus {N_{1}}\stackrel{\mathrm{F}}{\longrightarrow} N_{2}^{\perp} \oplus N_{2}= H_{\mathcal{A}} $$ 
		$$H_{\mathcal{A}} = {N_{1}^{\prime}}^{\perp}  \oplus {N_{1}^{\prime}}\stackrel{\mathrm{D}}{\longrightarrow} M_{2}^{\prime} \tilde \oplus N_{2}^{\prime}= H_{\mathcal{A}} $$ 
		w.r.t. which $\mathrm{F}, \mathrm{D}$ have matrices 
		\begin{center}
			$\left\lbrack
			\begin{array}{ll}
			\mathrm{F}_{1} & 0 \\
			0 & \mathrm{F}_{4} \\
			\end{array}
			\right \rbrack
			$,
			$\left\lbrack
			\begin{array}{ll}
			\mathrm{D}_{1} & 0 \\
			0 & \mathrm{D}_{4} \\
			\end{array}
			\right \rbrack
			,$	
		\end{center}
		respectively, where $\mathrm{F}_{1},\mathrm{D}_{1}$ are isomorphims, $ N_{1},N_{2}^{\prime}  $ are finitely generated and assume also that one of the following statements  hold:\\
		a) There exists some $\mathrm{J} \in B^{a}  (N_{2}, N_{1}^{\prime})  $ such that $N_{2} \cong \mathrm{Im}\mathrm{J}   $ and $\mathrm{Im}\mathrm{J}^{\perp} $ is finitely generated.\\
		b) There exists some $\mathrm{J}^{\prime} \in B^{a}  (N_{1}^{\prime}, N_{2} )   $ such that $N_{1}^{\prime} \cong \mathrm{Im}\mathrm{J}^{\prime}, (\mathrm{Im}\mathrm{J}^{\prime})^{\perp} $ is finitely generated.\\
		Then $ {\mathbf{M}_{\mathrm{C}}^{\mathcal{A}}} \in \mathcal{M} \Phi(H_{\mathcal{A}} \oplus H_{\mathcal{A}})  $ for some $\mathrm{C} \in B^{a}  (H_{\mathcal{A}})    .$	
	\end{theorem}
	\begin{remark}
		$\mathrm{Im}\mathrm{J}^{\perp} $ in part a) denotes the orthogonal complement of $\mathrm{Im}\mathrm{J}$ in $N_{1}^{\prime}$ and $\mathrm{Im}{\mathrm{J}^{\prime}}^{\perp}$ denotes the orthogonal complement of $\mathrm{Im}{\mathrm{J}^{\prime}}$ in $N_{2}.$ \\
		By \cite[Theorem 2.3.3 ]{MT}, if  $\mathrm{Im}\mathrm{J}$ is closed, then $ \mathrm{Im}\mathrm{J}  $ is indeed orthogonally complementable, so since in assumption a) above $\mathrm{Im}\mathrm{J} \cong N_{2}   $, it follows that $\mathrm{Im} \mathrm{J}   $ is closed, so $ N_{1}^{\prime} = \mathrm{Im}\mathrm{J} \oplus \mathrm{Im} \mathrm{J}^{\perp}  $. Similarly, in b) $N_{2}=\mathrm{Im}\mathrm{J}^{\prime} \oplus \mathrm{Im}{\mathrm{J}^{\prime}}^{\perp}.$
	\end{remark}
	\begin{proof}
		Suppose that b) holds, and consider the operator $\tilde{\mathrm{J}^{\prime}} = \mathrm{J}^{\prime}\mathrm{P}_{N_{1}^{\prime}}  $ where $ \mathrm{P}_{N_{1}^{\prime}}   $ denotes the orthogonal projection onto $   N_{1}^{\prime}.$ Then $\tilde{\mathrm{J}^{\prime}}   $ can be considered as a bounded adjointable operator on $H_{\mathcal{A}}   $ (as $N_{2}   $ is orthogonally complementable in
		$( H_{\mathcal{A}} ).$
		To simplify notation, we let $M_{2}=N_{2}^{\perp},M_{1}^{\prime} = {N_{1}^{\prime}}^{\perp}  $ and we let\\
		$\mathbf{M}_{\tilde{\mathrm{J}^{\prime}}}^{\mathcal{A}}=\mathbf{M}_{\tilde{\mathrm{J}^{\prime}}} .$ We claim then that w.r.t. the decomposition
		$$H_{\mathcal{A}} \oplus H_{\mathcal{A}}= (M_{1} \oplus H_{\mathcal{A}}) \tilde{\oplus} (N_{1} \oplus \lbrace 0 \rbrace )   $$ 
		$$\downarrow  \mathbf{M}_{\tilde{\mathrm{J}^{\prime}}}$$
		$$H_{\mathcal{A}} \oplus H_{\mathcal{A}}=((M_{2} \tilde{\oplus} \mathrm{Im}\mathrm{J}^{\prime} )\oplus M_{2}^{\prime}) \tilde{\oplus} ( {\mathrm{Im}\mathrm{J}^{\prime}}^{\perp} \oplus N_{2}^{\prime} ) ,$$
		$\mathbf{M}_{\tilde{\mathrm{J}^{\prime}}} $ has the matrix
		\begin{center}
			$\left\lbrack
			\begin{array}{ll}
			(\mathbf{M}_{\tilde{\mathrm{J}^{\prime}}})_{1} & (\mathbf{M}_{\tilde{\mathrm{J}^{\prime}}})_{2} \\
			(\mathbf{M}_{\tilde{\mathrm{J}^{\prime}}})_{3} & (\mathbf{M}_{\tilde{\mathrm{J}^{\prime}}})_{4} \\
			\end{array}
			\right \rbrack
			,$	
		\end{center}
		where $(\mathbf{M}_{\tilde{\mathrm{J}^{\prime}}})_{1}   $ is an isomorphism. To see this observe first that 
		$$(\mathbf{M}_{\tilde{\mathrm{J}^{\prime}}})_{1}=\sqcap_{(M_{2} \tilde{\oplus} \mathrm{Im}\mathrm{J}^{\prime} )\oplus M_{2}^{\prime}} {\mathbf{M}_{\tilde{\mathrm{J}^{\prime}}}}_{{\mid}_{M_{1} \oplus H_{\mathcal{A}}}} =$$
		\begin{center}
			$
			\left\lbrack
			\begin{array}{ll}
			\mathrm{F}_{\mid_{M_{1}}} & \tilde{\mathrm{J}^{\prime}} \\
			0 & \mathrm{D}\sqcap_{M_{1}^{\prime}} \\
			\end{array}
			\right \rbrack
			$
		\end{center}
		( as $\sqcap_{M_{2}^{\prime}} \mathrm{D}=\mathrm{D} \sqcap_{M_{1}^{\prime}} $ ), where $\sqcap_{(M_{2} \tilde{\oplus} \mathrm{Im}\mathrm{J}^{\prime} )\oplus M_{2}^{\prime}} $  denotes the projection onto\\
		$(M_{2} \tilde{\oplus} \mathrm{Im}\mathrm{J}^{\prime} ) \oplus M_{2}^{\prime}   $ along
		${\mathrm{Im}\mathrm{J}^{\prime}}^{\perp} ) \oplus N_{2}^{\prime}    $ and $\sqcap_{M_{1}^{\prime}}   $ denotes the projection onto $M_{1}^{\prime}   $ along $N_{1}^{\prime}   $. Clearly, $ (\mathbf{M}_{\tilde{\mathrm{J}^{\prime}}})_{1}  $ is onto $(M_{2} \tilde{\oplus} \mathrm{Im}\mathrm{J}^{\prime}\oplus M_{2}^{\prime}$. Now, if $ (\mathbf{M}_{\tilde{\mathrm{J}^{\prime}}})_{1}$ 
		$\left\lbrack
		\begin{array}{l}
		x  \\
		y  \\
		\end{array}
		\right \rbrack
		$ 
		$=$
		$\left\lbrack
		\begin{array}{l}
		0  \\
		0  \\
		\end{array}
		\right \rbrack
		$
		for some $x \in M_{1}, y \in H_{\mathcal{A}}   $, then $ \mathrm{D} \sqcap_{M_{1}^{\prime}} y=0 ,$ so $ y\in N_{1}^{\prime}  $ as $ \mathrm{D}_{\mid_{{M_{1}^{\prime}}}}  $ is bounded below. Also $\mathrm{F}x+ \tilde{\mathrm{J}^{\prime}}y=0  $. But, since $y \in N_{1}^{\prime},$ then $\tilde{\mathrm{J}^{\prime}}y=\mathrm{J}^{\prime}y   $, so we get $\mathrm{F}x+ {\mathrm{J}^{\prime}}y=0   $. Since $\mathrm{F}x \in  M_{2},\mathrm{J}^{\prime}y=N_{2}   $ and $M_{2} \cap  N_{2}=\lbrace 0 \rbrace  $, we get $ \mathrm{F}x= {\mathrm{J}^{\prime}}y=0   $. Since $\mathrm{F}_{\mid_{M_{1}}}   $ and $\mathrm{J}^{\prime}   $ are bounded below, we get $x=y=0   $. So $(\mathbf{M}_{\tilde{\mathrm{J}^{\prime}}})_{1}   $ is injective as well, thus an isomorphism.
		Recall next that $N_{1}\oplus \lbrace 0 \rbrace   $ and $\mathrm{Im}{\mathrm{J}^{\prime}}^{\perp} \oplus N_{2}^{\prime}  $ are finitely generated. By using the procedure of diagonalisation of $\mathbf{M}_{\tilde{\mathrm{J}^{\prime}}} $ as done in the proof of \cite[Lemma 2.7.10]{MT}, we obtain that $\mathbf{M}_{\tilde{\mathrm{J}^{\prime}}} \in \mathcal{M} \Phi (H_{\mathcal{A}} \oplus H_{\mathcal{A}})   .$\\
		Assume now that a) holds. Then there exists
		${\iota} \in  B^{a}(\mathrm{Im}\mathrm{J},N_{2})   $ s.t ${\iota}\mathrm{J}=id_{N_{2}}   .$\\
		Let $\widehat{\iota}={\iota}\mathrm{P}_{\mathrm{Im}\mathrm{J}}   $ where $\mathrm{P}_{\mathrm{Im}\mathrm{J}}   $ denote the orthogonal projection onto $ \mathrm{Im}\mathrm{J}  $. (notice that $\mathrm{Im}\mathrm{J}    $ is orthogonally complementable in $H_{\mathcal{A}}  $ since it is orthogonally complementable in $ N_{1}^{\prime} $ and $H_{\mathcal{A}}=N_{1}^{\prime}  \oplus {N_{1}^{\prime} }^{\perp}  ).$ Thus $\widehat{\iota} \in B^{a}  (H_{\mathcal{A}})   $. Consider
		$\mathbf{M}_{\widehat{\iota} }=\left\lbrack
		\begin{array}{ll}
		\mathrm{F} & {\widehat{\iota}}  \\
		0 & \mathrm{D} \\
		\end{array}
		\right \rbrack
		$.
		We claim that w.r.t. the decomposition
		$$H_{\mathcal{A}} \oplus H_{\mathcal{A}}= (M_{1} \oplus (M_{1}^{\prime} \tilde{\oplus} \mathrm{Im}\mathrm{J})) \tilde{\oplus} (N_{1} \oplus \mathrm{Im}\mathrm{J}^{\perp} ))   $$ 
		$$\downarrow  \mathbf{M}_{\widehat{\iota}}$$
		$$H_{\mathcal{A}} \oplus H_{\mathcal{A}}= (H_{\mathcal{A}} \oplus M_{2}^{\prime}) \tilde{\oplus} ( \lbrace 0 \rbrace \oplus N_{2}^{\prime})   ,$$ 
		$M_{{\iota}} $ has the matrix
		$\left\lbrack
		\begin{array}{ll}
		(\mathbf{M}_{\widehat{\iota}})_{1} & (\mathbf{M}_{\widehat{\iota}})_{2} \\
		(\mathbf{M}_{\widehat{\iota}})_{3} & (\mathbf{M}_{\widehat{\iota}})_{4} \\
		\end{array}
		\right \rbrack
		$,
		where $(\mathbf{M}_{\widehat{\iota}})_{1}   $ is an isomorphism. To see this, observe again that\\
		$(\mathbf{M}_{\widehat{\iota}})_{1} =\sqcap_{(H_{\mathcal{A}} {\oplus} M_{2}^{\prime} )} {\mathbf{M}_{\widehat{\iota}}}_{{\mid}_{M_{1} \oplus (M_{1}^{\prime} {\oplus} \mathrm{Im}\mathrm{J})}} = $
		$
		\left\lbrack
		\begin{array}{ll}
		\mathrm{F}_{{\mid}_{M_{1} }}  & {\widehat{\iota}} \\
		0 & \mathrm{D}\sqcap_{M_{1}^{\prime}} \\
		\end{array}
		\right \rbrack
		,$
		so $(\mathbf{M}_{\widehat{\iota}})_{1}   $ is obviously onto $H_{\mathcal{A}} \oplus M_{2}^{\prime}   .$\\
		Moreover, if $ (\mathbf{M}_{\widehat{\iota}})_{1}$ 
		$\left\lbrack
		\begin{array}{l}
		x  \\
		y  \\
		\end{array}
		\right \rbrack
		$ 
		$=$
		$\left\lbrack
		\begin{array}{l}
		0  \\
		0  \\
		\end{array}
		\right \rbrack
		$ for some $ x \in M_{1}   $ and $ y \in M_{1}^{\prime} \tilde{\oplus} \mathrm{Im}\mathrm{J}$, we get that $ \mathrm{D} \sqcap_{M_{1}^{\prime}}y=0    $, so $y \in \mathrm{Im}\mathrm{J}.$\\
		Hence $ \widehat{\iota}y={\iota}y  $, so, $\mathrm{F}x+\widehat{\iota}y=\mathrm{F}x+{\iota}y=0   $. Since $\mathrm{F}x \in M_{2}, {\iota}y \in N_{2}   $ and
		$M_{2} \cap N_{2} = \lbrace 0 \rbrace   $, we get $\mathrm{F}x={\iota}y=0   $. As $\mathrm{F}_{\mid_{M_{1}}}   $ and $ {\iota}  $ are bounded below, we deduce that $x=y=0$. So $  (\mathbf{M}_{\widehat{\iota}})_{1}  $, is also injective, hence
		an isomorphism. In addition, we recall that $N_{1} \oplus \mathrm{Im}\mathrm{J}^{\perp}   $ and $\lbrace 0 \rbrace \oplus N_{2}^{\prime}  $ are finitely generated, so by the same arguments as before, we deduce that $\mathbf{M}_{\widehat{\iota}}  \in \mathcal{M} \Phi(H_{\mathcal{A}} \oplus H_{\mathcal{A}})  $.
	\end{proof}
	\begin{remark}
		We know from the proofs of \cite[Theorem 2.2]{I} and \cite[Theorem 2.3]{I}, part $1) \Rightarrow 2) $ that since $$\mathrm{F} \in \mathcal{M} \Phi_{+} (H_{\mathcal{A}}),\mathrm{D} \in \mathcal{M} \Phi_{-} (H_{\mathcal{A}})   ,$$ we can find the decompositions 
		$$H_{\mathcal{A}} = M_{1} \tilde \oplus {N_{1}}\stackrel{\mathrm{F}}{\longrightarrow} N_{2}^{\perp} \oplus N_{2}= H_{\mathcal{A}} ,$$ 
		$$H_{\mathcal{A}} = {N_{1}^{\prime}}^{\perp} \oplus {N_{1}^{\prime}}\stackrel{\mathrm{D}}{\longrightarrow} M_{2}^{\prime} \tilde \oplus N_{2}^{\prime}= H_{\mathcal{A}} ,$$ 
		w.r.t. which $\mathrm{F}, \mathrm{D}$ have matrices 
		\begin{center}
			$\left\lbrack
			\begin{array}{ll}
			\mathrm{F}_{1} & 0 \\
			0 & \mathrm{F}_{4} \\
			\end{array}
			\right \rbrack
			$,
			$\left\lbrack
			\begin{array}{ll}
			\mathrm{D}_{1} & 0 \\
			0 & \mathrm{D}_{4} \\
			\end{array}
			\right \rbrack
			,$	
		\end{center}
		respectively, where $\mathrm{F}_{1},\mathrm{D}_{1} $  are isomorphisms, $N_{1}, N_{2}^{\prime}$ are finitely generated. However, in this theorem we have also the additional assumptions a) and b).	
	\end{remark}
	%
	\begin{remark}
		\cite[Theorem 3.2 ]{DDj}, part $(ii)\Rightarrow (i)$ follows as a direct consequence of Theorem \ref{t360} in the case when $X=Y=H,$ where $H$ is a Hilbert space. Indeed, if $\mathrm{F} \in \Phi_{+}(H), \mathrm{D} \in \Phi_{-}(H),$ $\ker \mathrm{D} $ and $\mathrm{Im} \mathrm{F}^{\bot}$ are isomorphic up to a finite dimensional subspace, then we may let $$M_{1}= \ker \mathrm{F}^{\bot}, N_{1}= \ker \mathrm{F}^{\bot} , {N_{2}}^{\bot}= \mathrm{Im} \mathrm{F} , N_{2} = \mathrm{Im}  \mathrm{F}^{\bot} ,  N_{1}^{\prime}= \ker \mathrm{D},$$ $$ M_{2}^{\prime}=\mathrm{Im}\mathrm{D}, N_{2}^{\prime}= \mathrm{Im} \mathrm{D}^{\bot},  N_{1}^{\prime}= \ker \mathrm{D}  .$$ Since $\ker \mathrm{D} $ and $\mathrm{Im} \mathrm{F}^{\bot} $ are isomorphic up to a finite dimensional subspace, by \cite[Definition 2.2 ]{DDj} this means that either the condition a) or the condition b) in Theorem \ref{t360} holds. By Theorem \ref{t360} it follows then that ${\mathbf{M}_{\mathrm{C}}} \in \Phi(H \oplus H).$\\	
	\end{remark}
	%
	Let $\tilde{W}(\mathrm{F},\mathrm{D})$ be the set of all $ \alpha \in Z(\mathcal{A})$ such that
	there exist decompositions
	$$H,_{\mathcal{A}} = M_{1} \tilde \oplus {N_{1}}\stackrel{\mathrm{F}-\alpha 1}{\longrightarrow} M_{2} \tilde \oplus N_{2}= H_{\mathcal{A}} ,$$
	$$H_{\mathcal{A}} ,= M_{1}^{\prime} \tilde \oplus {N_{1}^{\prime}}\stackrel{\mathrm{D}-\alpha 1}{\longrightarrow} M_{2}^{\prime} \tilde \oplus N_{2}^{\prime}= H_{\mathcal{A}} ,$$
	w.r.t. which $\mathrm{F}-\alpha 1,\mathrm{D}-\alpha 1$ have matrices
	\begin{center}
		$
		\left\lbrack
		\begin{array}{ll}
		(\mathrm{F}- \alpha 1)_{1} & 0 \\
		0 & (\mathrm{F}- \alpha 1)_{4} \\
		\end{array}
		\right \rbrack ,
		$
		$
		\left\lbrack
		\begin{array}{ll}
		(\mathrm{D}- \alpha 1)_{1} & 0 \\
		0 & (\mathrm{D}- \alpha 1)_{4} \\
		\end{array}
		\right \rbrack ,
		$
	\end{center}
	where $(\mathrm{F}- \alpha 1)_{1},(\mathrm{D}- \alpha 1)_{1}$ are isomorphisms, $N_{1}, N_{2}^{\prime}  $ are finitely generated submodules  and such that there are no closed submodules $\tilde{N_{2}} ,\tilde{\tilde{N_{2}}},\tilde{N_{1}^{\prime}} , \tilde{\tilde{N_{1}^{\prime}}}    $ with the property that $N_{2} \cong \tilde{N_{2}} ,N_{1}^{\prime} \cong \tilde{N_{1}^{\prime}}, \tilde{\tilde{N_{2}}}, \tilde{\tilde{N_{1}}}  $ $   $ are finitely generated and
	$$ (\tilde{N_{2}} \oplus \tilde{\tilde{N_{2}}} ) \cong  (\tilde{N_{2}}^{\prime} \oplus \tilde{\tilde{N_{2}}}^{\prime} ).$$
	Set $W(\mathrm{F},\mathrm{D})$ to be the set of all $\alpha \in Z(\mathcal{A}) $ such that there are no decompositions
	$$H_{\mathcal{A}} = M_{1} \tilde \oplus {N_{1}}\stackrel{\mathrm{F}-\alpha 1}{\longrightarrow} N_{2}^{\perp} \tilde \oplus N_{2}= H_{\mathcal{A}} ,$$
	$$H_{\mathcal{A}} = {N_{1}^{\prime}}^{\perp} \tilde \oplus {N_{1}^{\prime}}\stackrel{\mathrm{D}-\alpha 1}{\longrightarrow} M_{2}^{\prime} \tilde \oplus N_{2}^{\prime}= H_{\mathcal{A}} ,$$
	w.r.t. which $\mathrm{F}- \alpha 1,\mathrm{D}- \alpha 1    $ have matrices
	\begin{center}
		$
		\left\lbrack
		\begin{array}{ll}
		(\mathrm{F}- \alpha 1)_{1} & 0 \\
		0 & (\mathrm{F}- \alpha 1)_{4} \\
		\end{array}
		\right \rbrack ,
		$
		$
		\left\lbrack
		\begin{array}{ll}
		(\mathrm{D}- \alpha 1)_{1} & 0 \\
		0 & (\mathrm{D}- \alpha 1)_{4} \\
		\end{array}
		\right \rbrack ,
		$ 
	\end{center}
	where $(\mathrm{F}- \alpha 1)_{1},(\mathrm{D}- \alpha 1)_{1}   $, are isomorphisms $N_{1}, N_{2}^{\prime}  $
	are finitely generated and with the property that a) or b) in the Theorem \ref{t360} hold. Then we have the following corollary:
	\begin{corollary} \label{c310} 
		For given $ \mathrm{F} \in B^{a}  (H_{\mathcal{A}})  $ and $  \mathrm{D} \in B^{a}  (H_{\mathcal{A}}),$ 
		$$\tilde{W}(\mathrm{F},\mathrm{D})  \subseteq \bigcap_{\mathrm{C} \in B^{a}  (H_{\mathcal{A}})}   \sigma_{e}^{\mathcal{A}}({\mathbf{M}_{\mathrm{C}}^{\mathcal{A}}}) \subseteq W(\mathrm{F},\mathrm{D}) .$$	
	\end{corollary}
	\begin{theorem} \label{t311} 
		Suppose  $ {\mathbf{M}_{\mathrm{C}}^{\mathcal{A}}} \in \mathcal{M} \Phi_{-}(H_{\mathcal{A}} {\oplus} H_{\mathcal{A}}  )   $  for some  $ \mathrm{C} \in B^{a}(H_{\mathcal{A}})  .$ Then  $\mathrm{D} \in \mathcal{M} \Phi_{-}(H_{\mathcal{A}})  $  and in addition the following statement holds:\\
		Either  $ \mathrm{F} \in \mathcal{M} \Phi_{-}(H_{\mathcal{A}}) $  or there exists decompositions  
		$$H_{\mathcal{A}} \oplus H_{\mathcal{A}} = M_{1} \tilde \oplus {N_{1}}\stackrel{\mathrm{F}^{\prime}}{\longrightarrow}   M_{2} \tilde \oplus N_{2}= H_{\mathcal{A}} \oplus H_{\mathcal{A}},  $$
		$$H_{\mathcal{A}} \oplus H_{\mathcal{A}} = M_{1}^{\prime} \tilde \oplus {N_{1}^{\prime}}\stackrel{\mathrm{D}^{\prime}}{\longrightarrow}  M_{2}^{\prime} \tilde \oplus N_{2}^{\prime}= H_{\mathcal{A}} \oplus H_{\mathcal{A}},   $$  
		w.r.t. which  
		$\mathrm{F}^{\prime},\mathrm{D}^{\prime}  $   have the matrices
		$\left\lbrack
		\begin{array}{ll}
		\mathrm{F}_{1}^{\prime} & 0 \\
		0 & \mathrm{F}_{4}^{\prime} \\
		\end{array}
		\right \rbrack ,
		$
		$\left\lbrack
		\begin{array}{ll}
		\mathrm{D}_{1}^{\prime} & 0 \\
		0 & \mathrm{D}_{4}^{\prime} \\
		\end{array}
		\right \rbrack ,
		$  
		where  $ \mathrm{F}_{1}^{\prime},\mathrm{D}_{1}^{\prime}    $  are isomorphisms,  $ N_{2}^{\prime}   $  is finitely generated,  $N_{1},N_{2}, N_{1}^{\prime}  $  are closed, but \underline{not} finitely generated, and  $ M_{2} \cong  M_{1}^{\prime}, N_{2} \cong  N_{1}^{\prime}   .$\\	
	\end{theorem}
	\begin{proof}
		If  $ {\mathbf{M}_{\mathrm{C}}^{\mathcal{A}}} \in \mathcal{M} \Phi_{-}(H_{\mathcal{A}} {\oplus} H_{\mathcal{A}}  )     ,$  then there exists a decomposition
		$$H_{\mathcal{A}} \oplus H_{\mathcal{A}} = M_{1} \oplus {N_{1}}\stackrel{{\mathbf{M}_{\mathrm{C}}^{\mathcal{A}}}}{\longrightarrow} M_{2} \tilde \oplus N_{2}= H_{\mathcal{A}} \tilde \oplus H_{\mathcal{A}}  $$ 
		w.r.t. which  $ \mathbf{M}_{C}  $  has the matrix
		$\left\lbrack
		\begin{array}{ll}
		({\mathbf{M}_{\mathrm{C}}^{\mathcal{A}}} )_{1} & 0 \\
		0 &  ({\mathbf{M}_{\mathrm{C}}^{\mathcal{A}}} )_{4} \\
		\end{array}
		\right \rbrack ,
		$ 
		where  $ ({\mathbf{M}_{\mathrm{C}}^{\mathcal{A}}} )_{1}   $  is an isomorphism and  $ N_{2}  $  is finitely generated. By the part of \cite[Theorem 2.3]{I}, part $1) \Rightarrow 2)$ we may assume that  $ M_{1}=N_{1}^{\bot}  $ . Hence  $ \mathrm{F}_{\mid_{M_{1}}}^{\prime}  $  is adjointable. Since  $\mathrm{F}_{\mid_{M_{1}}}^{\prime}    $  can be viewed as an operator in  $B^{a}(M_{1},(\mathrm{D}^{\prime} \mathrm{C}^{\prime} )^{-1}(M_{2}))   $, as $M_{1}$ is orthogonally complementable, \\
		by \cite[Theorem 2.3.3.]{MT}, $ \mathrm{F}^{\prime}(M_{1})  $ is orthogonally complementable in $(\mathrm{D}^{\prime} \mathrm{C}^{\prime} )^{-1}(M_{2}).$ By the same arguments as in the proof of \cite[Theorem 2.2]{I} part $2)\Rightarrow 1)$ we deduce that there exists a chain of decompositions  
		$$M_{1} \tilde \oplus {N_{1}}\stackrel{\mathrm{F}^{\prime}}{\longrightarrow}  R_{1} \tilde \oplus {R_{2}}\stackrel{\mathrm{C}^{\prime}}{\longrightarrow}  \mathrm{C}^{\prime}(R_{1}) \tilde \oplus \mathrm{C}^{\prime}({R_{2}})\stackrel{\mathrm{D}^{\prime}}{\longrightarrow}  M_{2} \tilde \oplus {N_{2}}$$  
		w.r.t. which  $\mathrm{F}^{\prime},\mathrm{C}^{\prime},\mathrm{D}^{\prime}   $  have matrices
		$\left\lbrack
		\begin{array}{ll}
		\mathrm{F}_{1}^{\prime} & 0 \\
		0 & \mathrm{F}_{4}^{\prime} \\
		\end{array}
		\right \rbrack ,
		$
		$\left\lbrack
		\begin{array}{ll}
		\mathrm{C}_{1}^{\prime} & 0 \\
		0 & \mathrm{C}_{4}^{\prime} \\
		\end{array}
		\right \rbrack ,
		$ 
		$\left\lbrack
		\begin{array}{ll}
		\mathrm{D}_{1}^{\prime} & \mathrm{D}_{2}^{\prime} \\
		0 & \mathrm{D}_{4}^{\prime} \\
		\end{array}
		\right \rbrack ,
		$\\
		where  $\mathrm{F}_{1}^{\prime},\mathrm{C}_{1}^{\prime}, \mathrm{C}_{4}^{\prime}, \mathrm{D}_{1}^{\prime}  $  are isomorphisms. Hence  $ \mathrm{D}^{\prime}    $  has the matrix
		$\left\lbrack
		\begin{array}{ll}
		\mathrm{D}_{1}^{\prime} & 0 \\
		0 & \tilde{\mathrm{D}}_{4}^{\prime} \\
		\end{array}
		\right \rbrack ,
		$ 
		w.r.t. the decomposition
		$$H_{\mathcal{A}} \oplus H_{\mathcal{A}} =\mathrm{W}\mathrm{C}^{\prime}(R_{1})  \tilde \oplus \mathrm{W}\mathrm{C}^{\prime}(R_{2})\stackrel{\mathrm{D}^{\prime}}{\longrightarrow}  M_{2} \tilde \oplus {N_{2}} =H_{\mathcal{A}} \oplus H_{\mathcal{A}} ,$$  
		where  $ \mathrm{W}  $  is an isomorphism. It follows that  $ \mathrm{D}^{\prime} \in \mathcal{M} \Phi_{-}(H_{\mathcal{A}} \tilde \oplus H_{\mathcal{A}})  $,  as  $ N_{2}  $  is finitely generated. Hence  $\mathrm{D} \in \mathcal{M} \Phi_{-}(H_{\mathcal{A}} )   $  (by the same arguments as in the proof of Theorem \ref{t320}). Next, assume that  $\mathrm{F} \notin \mathcal{M} \Phi_{-}(H_{\mathcal{A}})   ,$ then\\
		
		$\mathrm{F}^{\prime} \notin \mathcal{M} \Phi_{-}(H_{\mathcal{A}} \oplus H_{\mathcal{A}})   .$ Therefore  $ R_{2}  $  can \underline{not} be finitely generated (otherwise  $\mathrm{F}^{\prime}   $  would be in  $ \mathcal{M} \Phi_{-}(H_{\mathcal{A}} \oplus H_{\mathcal{A}})    $ ). Now,  $R_{1} \cong \mathrm{W}\mathrm{C}^{\prime}(R_{1}),R_{2}= \mathrm{W}\mathrm{C}^{\prime}(R_{2}).$
	\end{proof}
	\begin{remark}
		In case of ordinary Hilbert spaces, \cite[Theorem 4.4 ]{DDj}  part  $2)\Rightarrow 3)   $ follows as a corollary from Theorem \ref{t311}. Indeed, suppose that  $\mathrm{D} \in  B(H)    $  and that  $\mathrm{F} \in  B(H)   $  (where $H$ is a Hilbert space). If  $\ker \mathrm{D} \prec {\mathrm{Im}\mathrm{F}^{\perp}}$,  this means by \cite[Remark 4.4 ]{DDj} that  $\dim $ $\ker  \mathrm{D} < \infty.   $ So, if (2) in \cite[Theorem 4.4 ]{DDj} holds, that is  $ {\mathbf{M}_{\mathrm{C}}} \in \Phi_{-}(H \oplus H )  $  for some  $ \mathrm{C} \in B(H)  $,  then by Theorem \ref{t311}  $ \mathrm{D} \in \Phi_{-}(H)  $  and either  $\mathrm{F} \in  \Phi_{-}(H)   $  or there exist decompositions
		$$ H \oplus H = M_{1} \tilde \oplus {N_{1}}\stackrel{\mathrm{F}^{\prime}}{\longrightarrow}  M_{2} \tilde \oplus N_{2}= H  \oplus H,$$
		$$ H \oplus H = M_{1}^{\prime} \tilde \oplus {N_{1}^{\prime}}\stackrel{\mathrm{D}^{\prime}}{\longrightarrow}  M_{2}^{\prime} \tilde \oplus N_{2}^{\prime}= H  \oplus H,$$
		which satisfy the conditions described in Theorem \ref{t311}. In particular  $ N_{2},N_{1}^{\prime}  $  are infinite dimensional whereas  $N_{2}^{\prime}   $  is finite dimensional. Suppose that  $\mathrm{F} \notin \Phi_{-}(H)    $  and that the decompositions above exist. Observe that   $\ker \mathrm{D}^{\prime}=\lbrace 0 \rbrace \oplus  \ker  \mathrm{D}.$  Hence, if dim   $\ker \mathrm{D} < \infty   $,  then dim $\ker \mathrm{D}^{\prime} < \infty   $ . Since  $ \mathrm{D}_{\mid_{M_{1}^{\prime}}}^{\prime}  $  is an isomorphism, by the same arguments as in the proof of \cite[Proposition 3.6.8 ]{MT} one can deduce that   $\ker \mathrm{D}^{\prime} \subseteq N_{1}^{\prime}   $ . Assume that   $\dim \ker  \mathrm{D} =\dim \ker  \mathrm{D}^{\prime}    < \infty $   and let  $\tilde{N_{1}}^{\prime}   $  be the orthogonal complement of   $\ker  \mathrm{D}^{\prime}   $ in  ${N_{1}}^{\prime}    ,$ that is  $N_{1}^{\prime}  = \ker  \mathrm{D}^{\prime} \oplus \tilde N_{1}^{\prime}   .$ Now, since   $\mathrm{Im} \mathrm{D}^{\prime}   $  is closed as  $\mathrm{D}^{\prime} \in \mathcal{M} \Phi_{-}(H \oplus H)   ,$  then  $\mathrm{D}_{\mid_{\tilde N_{1}^{\prime}}}^{\prime}   $  is an isomorphism. Since   $\dim N_{1}^{\prime}= \infty $  and   $\dim \ker  \mathrm{D}^{\prime} < \infty   $ , we have   $\dim  N_{1}^{\prime}= \infty   $ . Hence  $\mathrm{D}^{\prime} (\tilde N_{1}^{\prime} )   $  is infinite dimensional subspace of  $ N_{2}^{\prime}  .$ This is a contradiction since   $\dim N_{2}^{\prime}$ is finite. Thus, if  $ \mathrm{F} \notin  \Phi_{-}(H),$  we must have that $\ker  \mathrm{D}$ is infinite dimensional. Hence, we deduce, as a corollary, \cite[Theorem 4.4 ]{DDj} in case when  $X=Y=H ,$  where $H$ is a Hilbert space. In this case, part $(3b)$ in \cite[Theorem 4.4 ]{DDj}  could be reduced to the following statement: Either  $ \mathrm{F} \in  \Phi_{-}(H)   $  or   $\dim$ $\ker \mathrm{D}= \infty .$	
	\end{remark}
	\begin{theorem} \label{t313} 
		Let  $ \mathrm{F},\mathrm{D} \in B^{a}(H_{\mathcal{A}})  $  and suppose that  $\mathrm{D} \in \mathcal{M} \Phi_{-}(H_{\mathcal{A}})     $  and either  $\mathrm{F} \in \mathcal{M} \Phi_{-}(H_{\mathcal{A}})     $  or that there exist decompositions
		$$H_{\mathcal{A}} = M_{1} \tilde \oplus {N_{1}}\stackrel{\mathrm{F}}{\longrightarrow} N_{2}^{\bot} \tilde \oplus N_{2}= H_{\mathcal{A}}   ,$$ 
		$$H_{\mathcal{A}} = {N_{1}^{\prime}}^{\bot} \tilde \oplus {N_{1}^{\prime}}\stackrel{\mathrm{D}}{\longrightarrow} M_{2}^{\prime} \tilde \oplus N_{2}^{\prime}= H_{\mathcal{A}} ,$$ 
		w.r.t. which $\mathrm{F}, \mathrm{D}$ have the matrices
		$\left\lbrack
		\begin{array}{ll}
		\mathrm{F}_{1} & 0 \\
		0 & \mathrm{F}_{4} \\
		\end{array}
		\right \rbrack ,
		$
		$\left\lbrack
		\begin{array}{ll}
		\mathrm{D}_{1} & 0 \\
		0 & \mathrm{D}_{4} \\
		\end{array}
		\right \rbrack 
		,$  
		respectively, where $\mathrm{F}_{1},\mathrm{D}_{1}$ are isomorphisms $  N_{2}^{\prime} ,$  is finitely generated and that there exists some\\
		$ \iota \in B^{a} ( N_{2},N_{1}^{\prime} ) $  such that $\iota$ is an isomorphism onto its image in  $N_{1}^{\prime}   $ . Then  ${\mathbf{M}_{\mathrm{C}}^{\mathcal{A}}} \in \mathcal{M} \Phi_{-}(H_{\mathcal{A}}  \oplus H_{\mathcal{A}} )   $  for some  $ \mathrm{C} \in B^{a}(H_{\mathcal{A}})  .$	
	\end{theorem}
	\begin{proof}
		Since $\mathrm{Im} \iota$ is closed and  $\iota \in B^{a} ( N_{2},N_{1}^{\prime} ),\mathrm{Im} \iota$ is orthogonally complementable in $N_{1}^{\prime}$ by \cite[Theorem 2.3.3 ]{MT},  that is  $N_{1}^{\prime} = \mathrm{Im}$ $ \iota $ $ \oplus \tilde{N_{1}^{\prime}}   $  for some closed submodule $\tilde{N_{1}^{\prime}}.$\\
		Hence  $H_{\mathcal{A}}=\mathrm{Im} \iota \oplus \tilde N_{1}^{\prime} \oplus {N_{1}^{\prime}}^{\bot}   ,$  that is $\mathrm{Im}\iota $ is orthogonally complementable in  $ H_{\mathcal{A}}  .$ Also, there exists  $\mathrm{J} \in  B^{a}(\mathrm{Im} \iota, N_{2})  $  such that   $\mathrm{J}\iota=id_{N_{2}},\iota \mathrm{J} =id_{\mathrm{Im}{\iota}}.$ Let  $\mathrm{P}_{\mathrm{Im}{\iota}}   $  be the orthogonal projection onto $\mathrm{Im} \iota$ and set  $\mathrm{C}=\mathrm{J}\mathrm{P}_{\mathrm{Im} \iota}   .$ Then  $\mathrm{C} \in B^{a}(H_{\mathcal{A}})  .$ Moreover, w.r.t. the decomposition 
		$$H_{\mathcal{A}} \oplus H_{\mathcal{A}}= ( M_{1} \oplus (M_{1}^{\prime} \tilde \oplus \mathrm{Im}\iota)) \tilde \oplus (N_{1}  \oplus \tilde N_{1}^{\prime})\stackrel{{\mathbf{M}_{\mathrm{C}}^{\mathcal{A}}}}{\longrightarrow}$$ 
		$$  (H_{\mathcal{A}} \oplus M_{2}^{\prime} ) \tilde \oplus (\lbrace 0 \rbrace \oplus N_{2}^{\prime} )=H_{\mathcal{A}} \oplus H_{\mathcal{A}}, $$
		$ {\mathbf{M}_{\mathrm{C}}^{\mathcal{A}}} $ has the matrix
		$\left\lbrack
		\begin{array}{ll}
		({{\mathbf{M}_{\mathrm{C}}^{\mathcal{A}}}})_{1} &({{\mathbf{M}_{\mathrm{C}}^{\mathcal{A}}}})_{2} \\
		({{\mathbf{M}_{\mathrm{C}}^{\mathcal{A}}}}_{3}) & ({{\mathbf{M}_{\mathrm{C}}^{\mathcal{A}}}})_{4} \\
		\end{array}
		\right \rbrack ,
		$   
		where  $ ({{\mathbf{M}_{\mathrm{C}}^{\mathcal{A}}}})_{1}  $  is an isomorphism. This follows by the same arguments as in the proof of Theorem \ref{t360}. Using that  $ N_{2}^{\prime}  $  is finitely generated and proceeding further as in the proof of the above mentiond theorem, we reach the desired conclusion.
	\end{proof}
	\begin{remark}
		In the case of ordinary Hilbert spaces, \cite[Theorem 4.4]{DDj}  part  $(1)\Rightarrow (2)   $  can be deduced as a corollary from Theorem \ref{t313}. 
		Indeed, if $\mathrm{F}$ is closed and $\mathrm{D}\in \Phi_{-}(H)$, which gives that $\mathrm{Im} \mathrm{D}$ is closed also, then the pair of decompositions
		$$ H = (\ker \mathrm{F})^{\perp} \oplus \ker \mathrm{F}\stackrel{\mathrm{F}}{\longrightarrow} \mathrm{Im} \mathrm{F} \oplus \mathrm{Im} \mathrm{F}^{\perp} =H ,$$ 
		$$ H = (\ker \mathrm{F})^{\perp} \oplus \ker \mathrm{D}\stackrel{\mathrm{D}}{\longrightarrow} \mathrm{Im} \mathrm{D} \oplus \mathrm{Im} \mathrm{D}^{\perp} =H $$
		for $\mathrm{F}$ and $\mathrm{D}$, respectively, is one particular pair of decompositions that satisfies the hypotheses of Theorem \ref{t313} as long $ (\mathrm{Im} \mathrm{F})^{\perp} \preceq \ker \mathrm{D}$.
	\end{remark}
	Let $R(\mathrm{F},\mathrm{D})$ be the set of all $\alpha \in Z(\mathcal{A})$ such that 	there exists no decompositions
	$$H_{\mathcal{A}} = M_{1} \tilde \oplus {N_{1}}\stackrel{\mathrm{F}-\alpha \mathrm{I}}{\longrightarrow}   {N_{2}}^{\perp} \tilde \oplus N_{2}= H_{\mathcal{A}},$$  
	$$H_{\mathcal{A}} = {N_{2}^{\prime}}^{\perp} \tilde \oplus {N_{1}^{\prime}}\stackrel{\mathrm{D}-\alpha \mathrm{I}}{\longrightarrow}   M_{2}^{\prime} \tilde \oplus N_{2}^{\prime}= H_{\mathcal{A}} $$
	that satisfy the hypotheses of the Theorem \ref{t313}. 
	Set  $R^{\prime}(\mathrm{F},\mathrm{D})$ to be the set of all $\alpha \in Z(\mathcal{A}) $ such that there exist no decompositions
	$$ H_{\mathcal{A}} \oplus  H_{\mathcal{A}} = M_{1} \tilde \oplus {N_{1} }\stackrel{\mathrm{F}^{\prime}-\alpha \mathrm{I}}{\longrightarrow}  M_{2} \tilde \oplus N_{2}= H_{\mathcal{A}}  \oplus  H_{\mathcal{A}}  , $$
	$$ H_{\mathcal{A}}  \oplus  H_{\mathcal{A}} = M_{1}^{\prime} \tilde \oplus {N_{1}^{\prime} }\stackrel{\mathrm{D}^{\prime}-\alpha \mathrm{I}}{\longrightarrow}  M_{2}^{\prime} \tilde \oplus N_{2}^{\prime}= H_{\mathcal{A}} \oplus  H_{\mathcal{A}}   $$
	that satisfy the hypotheses of the Theorem \ref{t311}.\\
	Then we have the following corollary: 
	\begin{corollary} \label{c314}
		Let $\mathrm{F},\mathrm{D} \in B^{a}(H_{\mathcal{A}})$. Then 
		$$\sigma_{{re}}^{\mathcal{A}}(\mathrm{D})\cup (\sigma_{{re}}^{\mathcal{A}}(\mathrm{F}) \cap R^{\prime}(\mathrm{F},\mathrm{D})) \subseteq \bigcap_{{\mathrm{C} \in  B^{a}(H_{\mathcal{A}})}}
		\sigma_{{re}}^{\mathcal{A}}({\mathbf{M}_{\mathrm{C}}^{\mathcal{A}}}) \subseteq \sigma_{{re}}^{\mathcal{A}}(\mathrm{D})\cup (\sigma_{{re}}^{\mathcal{A}}(\mathrm{F}) \cap R(\mathrm{F},\mathrm{D}))$$	\\
	\end{corollary}
	\begin{theorem} \label{t316} 
		Let ${\mathbf{M}_{\mathrm{C}}^{\mathcal{A}}} \in \mathcal{M}\Phi_{+}(H_{\mathcal{A}} \oplus H_{\mathcal{A}}).$ Then $\mathrm{F}^{\prime} \in \mathcal{M}\Phi_{+}(H_{\mathcal{A}} \oplus H_{\mathcal{A}})  $ and either $\mathrm{D} \in \mathcal{M}\Phi_{+}(H_{\mathcal{A}}) $ or there exist decompositions 
		$$ H_{\mathcal{A}} \oplus H_{\mathcal{A}} = M_{1} \tilde \oplus {N_{1}}\stackrel{\mathrm{F}^{\prime}}{\longrightarrow} M_{2} \tilde \oplus N_{2}=  H_{\mathcal{A}} \oplus H_{\mathcal{A}},$$
		$$  H_{\mathcal{A}} \oplus H_{\mathcal{A}} = M_{1}^{\prime} \tilde \oplus {N_{1}^{\prime}}\stackrel{\mathrm{D}^{\prime}}{\longrightarrow} M_{2}^{\prime} \tilde \oplus N_{2}^{\prime}=  H_{\mathcal{A}} \oplus H_{\mathcal{A}} ,$$ 
		w.r.t. which $ \mathrm{F}^{\prime},\mathrm{D}^{\prime}  $ have matrices
		$\left\lbrack
		\begin{array}{ll}
		\mathrm{F}_{1}^{\prime} & 0 \\
		0 & \mathrm{F}_{4}^{\prime} \\
		\end{array}
		\right \rbrack
		,$  
		$\left\lbrack
		\begin{array}{ll}
		\mathrm{D}_{1}^{\prime} & 0 \\
		0 & \mathrm{D}_{4}^{\prime} \\
		\end{array}
		\right \rbrack
		,$ 
		respectively, where $\mathrm{F}_{1}^{\prime}, \mathrm{D}_{1}^{\prime}$ are isomorphisms, $M_{2}\cong  M_{1}^{\prime} $ and $N_{2}\cong  N_{1}^{\prime},$ $N_{1}$ is finitely generated and $N_{2}, N_{1}^{\prime} $ are closed, but not finitely generated.  
	\end{theorem}
	\begin{proof}
		Since ${\mathbf{M}_{\mathrm{C}}^{\mathcal{A}}} \in \mathcal{M}\Phi_{+}(H_{\mathcal{A}} \oplus H_{\mathcal{A}}),$ there exists an $ \mathcal{M}\Phi_{+} $ decomposition for ${\mathbf{M}_{\mathrm{C}}^{\mathcal{A}}},$
		$$H_{\mathcal{A}} \oplus H_{\mathcal{A}} = M_{1} \tilde \oplus {N_{1}}\stackrel{{\mathbf{M}_{\mathrm{C}}^{\mathcal{A}}}}{\longrightarrow} M_{2}^{\prime} \tilde \oplus N_{2}^{\prime}= H_{\mathcal{A}} \oplus H_{\mathcal{A}} ,$$
		so $N_{1}$ is finitely generated. By the proof of \cite[Theorem 2.7.6 ]{MT}, we may assume that $M_{1}=N_{1}^{\perp}.$ Hence ${\mathrm{F}^{\prime}}_{{\mid}_{{M}_{1}}}   $, is adjointable. As in the proof of Lemma \ref{l210} and  Theorem \ref{t320} we may consider a chain of decompositions
		$$ H_{\mathcal{A}} \oplus H_{\mathcal{A}} = M_{1} \tilde \oplus {N_{1}}\stackrel{\mathrm{F}^{\prime}}{\longrightarrow} R_{1} \tilde \oplus R_{2}\stackrel{\mathrm{C}^{\prime}}{\longrightarrow} \mathrm{C}^{\prime}(R_{1}) \tilde \oplus \mathrm{C}^{\prime}(R_{2})\stackrel{\mathrm{D}^{\prime}}{\longrightarrow}  M_{2}^{\prime} \tilde \oplus M_{2}^{\prime}= H_{\mathcal{A}} \oplus H_{\mathcal{A}}$$ 
		w.r.t. which $\mathrm{F}^{\prime},\mathrm{C}^{\prime},\mathrm{D}^{\prime}$ have matrices
		$\left\lbrack
		\begin{array}{ll}
		\mathrm{F}_{1}^{\prime} & 0 \\
		0 & \mathrm{F}_{4}^{\prime} \\
		\end{array}
		\right \rbrack
		,$  
		$\left\lbrack
		\begin{array}{ll}
		\mathrm{C}_{1}^{\prime} & 0 \\
		0 & \mathrm{C}_{4}^{\prime} \\
		\end{array}
		\right \rbrack
		$  
		and 
		$\left\lbrack
		\begin{array}{ll}
		\mathrm{D}_{1}^{\prime} & \mathrm{D}_{2}^{\prime} \\
		0 & \mathrm{D}_{4}^{\prime} \\
		\end{array}
		\right \rbrack
		,$
		respectively, where $\mathrm{F}_{1}^{\prime},\mathrm{C}_{1}^{\prime},\mathrm{C}_{4}^{\prime},\mathrm{D}_{1}^{\prime}$ are isomorphisms. Then we can proceed in the same way as in the proof of Theorem \ref{t311}.
	\end{proof}
	\begin{remark}
		In the case of Hilbert spaces, the implication $(2)\Rightarrow (3)$ in\\
		\cite[Theorem 4.6]{DDj} follows as a corollary of Theorem \ref{t316}. Indeed, for the implication $(2)\Rightarrow (3b) $, we may proceed as follows: Since $\mathrm{Im}(\mathrm{F})^{0} \cong \mathrm{Im}(\mathrm{F})^{\bot}$ and $(\ker \mathrm{D})^{\prime} \cong \ker \mathrm{D}$ when one considers Hilbert spaces, then by \cite[Remark 4.3]{DDj}, $(\mathrm{Im} \mathrm{F})^{0} \prec (\ker \mathrm{D})^{\prime}  $ means simply that  $\dim \mathrm{Im} \mathrm{F}^{\bot}<\infty   $ whereas $\dim \ker \mathrm{D}=\infty$. If in addition $\mathrm{D} \notin \Phi_{+}(H)$, then $\mathrm{D}^{\prime} \notin \Phi_{+}(H \oplus H)  $. Now, if $\dim \mathrm{Im}(\mathrm{F})^{\bot} < \infty$, then $\dim \ker \mathrm{D}=\infty ,$ and $ \mathrm{F} \in \Phi (H)$ as $\mathrm{F} \in \Phi_{+}(H)   $ and $\dim \mathrm{Im}(\mathrm{F})^{\perp} <\infty .$ Then $  \mathrm{F}^{\prime} \in \Phi (H \oplus H)  ,$ so by \cite[Lemma 2.16 ]{I} $N_{2}$ must be finitely generated. Thus $N_{1}^{\prime}$ must be finitely generated being isomorphic to $N_{2}$.	By the same arguments as earlier, we have that $\ker \mathrm{D}^{\prime} \cong \ker \mathrm{D}   $ and $\ker \mathrm{D}^{\prime} \subseteq N_{1}^{\prime}$. Since we consider Hilbert spaces now, the fact that $N_{1}^{\prime}$ is finitely generated means actually that $N_{1}^{\prime}$ is finite dimensional. Hence $\ker \mathrm{D}^{\prime} $ must be finite dimensional, so $\dim \ker \mathrm{D}= \dim \ker \mathrm{D}^{\prime}  < \infty.$ This is in a contradiction to $\mathrm{Im} \mathrm{F}^{\bot} \prec \ker \mathrm{D} $. So, in the case of Hilbert spaces, if ${\mathbf{M}_{\mathrm{C}}} \in \Phi_{+}(H \oplus H)   $, from Theorem \ref{t316} it follows that $\mathrm{F} \in \Phi_{+}(H)   $ and either $\mathrm{D} \in \Phi_{+}(H)   $ or $\mathrm{Im} \mathrm{F}^{\bot}   $ is infinite dimensional.
	\end{remark}
	\begin{theorem} \label{t318}
		Let $\mathrm{F} \in \mathcal{M}\Phi_{+}(H_{\mathcal{A}}) $ and suppose that either $\mathrm{D} \in \mathcal{M}\Phi_{+}(H_{\mathcal{A}})    $ or that there exist decompositions 
		$$ H_{\mathcal{A}} = M_{1} \tilde \oplus {N_{1}}\stackrel{\mathrm{F}}{\longrightarrow} N_{2}^{\bot} \tilde \oplus N_{2}= H_{\mathcal{A}} ,$$
		$$ H_{\mathcal{A}} = {N_{1}^{\prime}}^{\bot} \tilde \oplus N_{1}^{\prime}\stackrel{\mathrm{D}}{\longrightarrow} M_{2}^{\prime} \tilde \oplus N_{2}^{\prime}= H_{\mathcal{A}}$$
		w.r.t. which $\mathrm{F},\mathrm{D}$ have matrices 
		$\left\lbrack
		\begin{array}{ll}
		\mathrm{F}_{1} & 0 \\
		0 & \mathrm{F}_{4} \\
		\end{array}
		\right \rbrack
		,$
		$\left\lbrack
		\begin{array}{ll}
		\mathrm{D}_{1} & 0 \\
		0 & \mathrm{D}_{4} \\
		\end{array}
		\right \rbrack
		,$
		respectively, where $\mathrm{F}_{1},\mathrm{D}_{1}$ are isomorphisms, $N_{1}^{\prime}$ is finitely generated and in addition there exists some\\
		$\iota \in B^{a} (N_{1}^{\prime},N_{2})$ such that $\iota$ is an isomorphism onto its image. Then $${\mathbf{M}_{\mathrm{C}}^{\mathcal{A}}} \in \mathcal{M}\Phi_{+}(H_{\mathcal{A}} \oplus H_{\mathcal{A}}),$$ for some $\mathrm{C} \in B^{a}(H_{\mathcal{A}})$.
	\end{theorem}
	\begin{proof}
		Let  $\mathrm{C}=\mathrm{P}_{N_{1}^{\prime}}  \iota$ where $\mathrm{P}_{N_{1}^{\prime}}$ denotes the orthogonal projection onto $N_{1}^{\prime},$ then apply similar arguments as in the proof of Theorem \ref{t360} and Theorem \ref{t313} 
	\end{proof}
	\begin{remark}
		The implication $(1)\Rightarrow (2)$ in \cite[Theorem 4.6]{DDj} in case of Hilbert spaces could also be deduced as a corollary from \ref{t318}. Indeed, if $ \mathrm{Im} \mathrm{D} $ is closed, then $ \mathrm{D} $ is an isomorphism from $ \ker \mathrm{D}^{\bot} $ onto $ \mathrm{Im} \mathrm{D}   $. Moreover, if $ \mathrm{F} \in \Phi_{+}(H) $, then $ \mathrm{F} $ is also an isomorphism from $  \ker \mathrm{F}^{\bot} $ onto $\mathrm{Im} \mathrm{F}$ and $\dim \ker \mathrm{F} < \infty $. If in addition $\ker \mathrm{D} \preceq \mathrm{Im} \mathrm{F}^{\bot}   $, then the pair of decompositions
		$$H=\ker \mathrm{F}^{\bot} \oplus \ker \mathrm{F}\stackrel{\mathrm{F}}{\longrightarrow} \mathrm{Im} \mathrm{F} \oplus \mathrm{Im} \mathrm{F}^{\bot}=H ,$$
		$$H=\ker \mathrm{D}^{\bot} \oplus \ker \mathrm{D}\stackrel{\mathrm{D}}{\longrightarrow} \mathrm{Im} \mathrm{D} \oplus \mathrm{Im} \mathrm{D}^{\bot}=H $$
		is one particular pair of decompositions that satisfies the hypotheses of Theorem \ref{t318}.	
	\end{remark} 
	%
	%
	Let $L^{\prime}(\mathrm{F},\mathrm{D})$ be the set of all $\alpha \in Z(\mathcal{A})$ such that there exist no decompositions 
	$$H_{\mathcal{A}} \oplus H_{\mathcal{A}} = M_{1} \tilde \oplus N_{1}\stackrel{\mathrm{F}^{\prime}-\alpha \mathrm{I}}{\longrightarrow} M_{2} \tilde \oplus N_{2}= H_{\mathcal{A}} \oplus H_{\mathcal{A}} ,$$
	$$H_{\mathcal{A}} \oplus H_{\mathcal{A}} = M_{1}^{\prime} \tilde \oplus N_{1}^{\prime}\stackrel{\mathrm{D}^{\prime}-\alpha \mathrm{I}}{\longrightarrow} M_{2}^{\prime} \tilde \oplus N_{2}^{\prime}= H_{\mathcal{A}} \oplus H_{\mathcal{A}} ,$$
	for $\mathrm{F}^{\prime}-\alpha \mathrm{I}, \mathrm{D}^{\prime}-\alpha \mathrm{I}$ respectively, which satisfy the hypotheses of Theorem \ref{t316}.\\
	Set $L(\mathrm{F},\mathrm{D})$ to be the set of all $\alpha \in Z(\mathcal{A})$ such that there exist no decompositions
	$$H_{\mathcal{A}} = M_{1} \tilde \oplus N_{1}\stackrel{\mathrm{F}-\alpha 1}{\longrightarrow} N_{2}^{\bot} \tilde \oplus N_{2}= H_{\mathcal{A}} ,$$
	$$H_{\mathcal{A}} = {N_{1}^{\prime}}^{\bot} \tilde \oplus N_{1}^{\prime}\stackrel{\mathrm{D}-\alpha 1}{\longrightarrow} M_{2}^{\prime} \tilde \oplus N_{2}^{\prime}= H_{\mathcal{A}},$$
	for $\mathrm{F}-\alpha 1, \mathrm{D}-\alpha 1 $ respectively which satisfy the hypotheses of Theorem \ref{t318}.\\
	Then we have the following corollary:
	\begin{corollary} \label{c320} 
		Corollary: Let $\mathrm{F},\mathrm{D} \in B^{a} (H_{\mathcal{A}})$. Then 
		$$\sigma_{{le}}^{\mathcal{A}}(\mathrm{F})\cup (\sigma_{{le}}^{\mathcal{A}}(\mathrm{D}) \cap L^{\prime}(\mathrm{F},\mathrm{D})) \subseteq \bigcap_{{\mathrm{C} \in  B^{a}(H_{\mathcal{A}})}}
		\sigma_{{le}}^{\mathcal{A}}({\mathbf{M}_{\mathrm{C}}^{\mathcal{A}}}) \subseteq \sigma_{{le}}^{\mathcal{A}}(\mathrm{F})\cup (\sigma_{{le}}^{\mathcal{A}}(\mathrm{D}) \cap L(\mathrm{F},\mathrm{D}))$$	
		
	\end{corollary}
\textbf{Acknowledgements:}	
		First of all, I am grateful to Professor Dragan S. Djordjevic for suggesting the research topic of the paper and for introducing to me the relevant reference books and papers. In addition,  I am especially grateful to my supervisors, Professor Vladimir M. Manuilov and Professor  Camillo Trapani, for careful reading of my paper and for detailed comments and suggestions which led to the improved presentation of the paper.


\end{document}